\documentclass[reqno,12pt]{amsart}
\usepackage{etoolbox}

\usepackage{fixltx2e,amssymb,indentfirst,xspace,bm}
\usepackage[retainorgcmds]{IEEEtrantools}
\usepackage{framed,paralist} 
\usepackage{mathtools}

\mathtoolsset{centercolon}

	\usepackage[dayofweek]{datetime}
	\usepackage{currfile}
\usepackage[normalem]{ulem}

\usepackage{xspace}

\usepackage[dvipsnames]{xcolor}

\usepackage[pagebackref=true, colorlinks=true, citecolor=blue, pdfborder={0 0 .1},breaklinks]{hyperref}

\usepackage[initials,nobysame]{amsrefs}

\usepackage{enumitem}

\makeatletter
\def\@wraptoccontribs#1#2{}

\@mparswitchfalse
\makeatother

\usepackage{cleveref}

\newbool{bScratchWork}

\newcommand{\Comment}[1]{{\color{Brown}#1}}
\newcommand{\ScratchWork}[1]{
 \ifbool{bScratchWork}
 {\begin{quote}\Comment{\footnotesize
 \medskip

 \noindent\textbf{Scratch work (not for submission)}:

 \noindent#1}
 \end{quote}
 }
 {
 }
 }

\makeatletter
\newcommand{\UWave}[2][blue]{\bgroup \markoverwith{\textcolor{#1}{\lower3.5\p@\hbox{\sixly \char58}}}\ULon{#2}}
\makeatother
\newcommand{\SOut}[2][red]{\bgroup\markoverwith {\textcolor{#1}{\rule[.45ex]{2pt}{.1ex}}}\ULon{#2}}

\newcommand{\highlight}[2][yellow]{\bgroup\markoverwith {\textcolor{#1}{\rule[-.2em]{2pt}{1.2em}}}\ULon{#2}}

 {\left\{\begin{IEEEeqnarraybox}[\relax][c]{#1}}%
 {\end{IEEEeqnarraybox}\right.}

\newcommand{\loc}{{\scriptstyle{\mathrm{loc}}}}

\renewcommand{\epsilon}{\varepsilon}
\newcommand{\eps}{\epsilon}

\renewcommand{\leq}{\leqslant}
\renewcommand{\geq}{\geqslant}
\renewcommand{\le}{\leqslant}
\renewcommand{\ge}{\geqslant}

\newcommand{\R}{\mathbb{R}}

\newcommand{\N}{\mathbb{N}}

\newif\iftextstyle
\textstyletrue
\everydisplay\expandafter{\the\everydisplay\textstylefalse}

\DeclarePairedDelimiterX{\norm}[1]{\lVert}{\rVert}{#1}

\DeclarePairedDelimiterX{\abs}[1]{\lvert}{\rvert}{#1}

\numberwithin{equation}{section}
\allowdisplaybreaks

\newtheorem{theorem}{Theorem}[section]

\newtheorem{lemma}[theorem]{Lemma}
\newtheorem{prop}[theorem]{Proposition}

\newtheorem{corollary}[theorem]{Corollary}

\newtheorem*{theorem*}{Theorem}
\newtheorem*{lemma*}{Lemma}
\newtheorem*{proposition*}{Proposition}
\newtheorem*{corollary*}{Corollary}

\newtheorem{definition}[theorem]{Definition}

\newtheorem{remark}[theorem]{Remark}
\newtheorem*{remark*}{Remark}

\newtheoremstyle{cases}
 {\smallskipamount}
 {\smallskipamount}
 {}
 {\parindent}
 {\itshape}
 {.}
 {.5em}
 {\thmnote{#1 #2: #3}}

\theoremstyle{cases}

\newcommand{\Ignore}[1]{}
\newcommand{\iny}{\ensuremath{\infty}}
\DeclareMathOperator{\dv}{div} %
\DeclareMathOperator{\curl}{curl} %
\DeclareMathOperator*{\supp}{supp} %

\newcommand{\real}{\mathbb{R}}

\DeclarePairedDelimiterX\set[1]{\{}{\}}{#1}

\newcommand{\n}{\bm{n}}

\DeclarePairedDelimiterX\innp[2]{(}{)}{#1,#2}

\newcommand{\e}{\bm{\mathrm{e}}}

\usepackage{tcolorbox}
\usepackage{tikz}
\tcbuselibrary{skins}
\usetikzlibrary{shadings}
\tcbset{
 myimage/.style={
 enhanced,
 overlay={
 \begin{scope}[shift={([xshift=1mm, yshift=7mm]frame.north west)}]
 \end{scope}}}}

\tcbset{
 skin=enhanced,
 fonttitle=\bfseries,
 interior style={white},
 segmentation style={black,solid,opacity=0.2,line width=1pt}}

\newtcolorbox{TitledBox}[2][]{
 myimage, 
 coltitle=black, 
 colbacktitle=white, 
 title=My title,
 attach boxed title to top center={
 yshift=-3mm,
 yshifttext=-1mm},
 attach boxed title to top left={
 xshift=1cm,
 yshift=-2mm},
 boxed title style={
 size=small},
 title={#2},#1}

\crefname{corollary}{Corollary}{Corollaries} 
\Crefname{corollary}{Corollary}{Corollaries} 
									
\crefname{lemma}{Lemma}{Lemmas}	 

\crefname{section}{Section}{Sections}
\Crefname{section}{Section}{Sections}

\crefname{appendix}{Appendix}{Appendices}
\Crefname{appendix}{Appendix}{Appendices}

\crefname{theorem}{Theorem}{Theorems}
\Crefname{theorem}{Theorem}{Theorems}

\crefname{prop}{Proposition}{Propositions}
\Crefname{prop}{Proposition}{Propositions}

\crefname{conj}{Conjecture}{Conjectures}
\Crefname{conj}{Conjecture}{Conjectures}

\crefname{definition}{Definition}{Definitions}
\Crefname{definition}{Definition}{Definitions}

\crefname{remark}{Remark}{Remarks}
\Crefname{remark}{Remark}{Remarks}

\crefname{assumption}{Assumption}{Assumptions}
\Crefname{assumption}{Assumption}{Assumptions}

\crefformat{equation}{(#2#1#3)}
\crefrangeformat{equation}{(#3#1#4) through (#5#2#6)}
\crefmultiformat{equation}
 {(#2#1#3)}%
 { and~(#2#1#3)}
 {, (#2#1#3)}
 { and~(#2#1#3)}

\newcommand{\vare}{\varepsilon}
\newcommand{\lo}{{\scriptstyle{\mathrm{loc}}}}
\newcommand{\tang}{{\scriptstyle{\mathrm{tan}}}}

\newcommand{\dd}{\,\mathrm{d}}
\newcommand{\Ker}{\mathrm{Ker }\,}
\newcommand{\ds}{\displaystyle}
\renewcommand{\hat}{\widehat}

\date\today

\begin{document}

\title
	[Navier BCs, 3D multiply connected]
	{Large time behavior for the 3D Navier-Stokes with Navier boundary conditions}
	
\author[Kelliher, Lacave, Lopes Filho, Nussenzveig Lopes, Titi]{James P. Kelliher$^1$, Christophe Lacave$^2$, Milton C. Lopes Filho$^3$, Helena J. Nussenzveig Lopes$^3$, and
Edriss S. Titi$^{4,}$\!$^{5,}$\!$^6$}
\address{$^1$ Department of Mathematics, University of California, Riverside, 900 University Ave., Riverside, CA 92521, U.S.A.}
\address{$^2$ Univ. Savoie Mont Blanc, CNRS, LAMA, ISTerre, 73000 Chamb\'ery, France}
\address{$^3$ Instituto de Matematica, Universidade Federal do Rio de Janeiro, Caixa Postal 68530, 21941-909, Rio de Janeiro, RJ, Brazil}
\address{$^4$ Department of Applied Mathematics and Theoretical Physics, University of Cambridge, Cambridge CB30WA, UK}
\address{$^5$ Department of Mathematics, Texas A\&M University, College Station, TX 77843, USA}
\address{$^6$ Department of Computer Science and Applied
Mathematics, Weizmann Institute of Science, Rehovot 76100, Israel}

\email{kelliher@math.ucr.edu}
\email{Christophe.Lacave@univ-smb.fr}
\email{mlopes@im.ufrj.br}
\email{hlopes@im.ufrj.br}
\email{titi@math.tamu.edu and Edriss.Titi@maths.cam.ac.uk}

\begin{abstract}
We study the three-dimensional incompressible Navier-Stokes equations in a smooth bounded domain $\Omega$ with initial velocity $u_0$ square-integrable, divergence-free and tangent to $\partial \Omega$. We supplement the equations with the Navier friction boundary conditions $u \cdot \n = 0$ and $[(2Su)\n + \alpha u]_{\tang} = 0$, where $\n$ is the unit exterior normal to $\partial \Omega$, $Su = (Du + (Du)^t)/2$, $\alpha \in C^0(\partial\Omega)$ is the boundary friction coefficient and $[\cdot]_{\tang}$ is the projection of its argument onto the tangent space of $\partial \Omega$. We prove global existence of a weak Leray-type solution to the resulting initial-boundary value problem and exponential decay in energy norm of these solutions when friction is positive. We also prove exponential decay if friction is non-negative and the domain is not a solid of revolution. These two results are well known in the case of Dirichlet boundary condition, but, even if they have been implicitly used for the Navier boundary conditions, the comprehensive analysis is not available in the literature. After carefully studying the Stokes semigroup for such a boundary condition, we use the Galerkin method for existence, Poincar\'{e}-type inequalities, with suitable adaptations to account for the differential geometry of the boundary, and a novel integral Gronwall-type inequality. In addition, in the frictionless case $\alpha = 0$, we prove convergence of the solution to a steady rigid rotation, if the domain is a solid of revolution.
\end{abstract}

\maketitle

\section{Introduction} \label{Int}

Let $\Omega \subset \R^3$ be a bounded, connected, open set with smooth boundary $\partial \Omega$. Fix $\nu > 0$.

We consider the following initial-value problem for the incompressible Navier-Stokes equations with viscosity $\nu$ and Navier friction boundary conditions:
\begin{equation}\label{NSwNBC}
\left\{
\begin{array}{ll}
\partial_t u + (u \cdot \nabla) u = -\nabla p + \nu \Delta u & \text{ in } (0,+\infty) \times \Omega, \\
\dv u = 0 & \text{ in } [0,+\infty) \times \Omega, \\
u \cdot \n = 0 \text{ and } [(2Su)\n + \alpha u]_{\tang} = 0 &\text{ on } (0,+\infty) \times \partial \Omega,\\
u(0,\cdot) = u_0 & \text{ in } \Omega.
\end{array}
\right.
\end{equation}
Above, $u:[0,+\infty) \times \Omega \to \R^3$ is the velocity field, $p$ is the scalar pressure, $Su := [Du + (Du)^t]/2$ is the symmetric part of the Jacobian of $u$, $\alpha: \partial \Omega \to \R$ is a given material friction coefficient, $\n$ refers to the outward unit normal vector to $\partial \Omega$ and the subscript `$\mathrm{tan}$' corresponds to the orthogonal projection onto the tangent space to $\partial \Omega$.

The Navier friction condition was first introduced in 1827 by C. Navier, see \cite{Navier1827}. This boundary condition is often considered in the presence of rough boundaries, see \cites{G-VM10,JM01} and references therein, or in flows with a turbulent layer next to the boundary, see \cite{JLS04}. Physically, it expresses a rough balance between the rate of fluid slip and friction stress at the boundary. In addition, allowing for slip may be regarded as a relaxation of the no-slip boundary condition, which creates stiffness in computational modeling of high-Reynolds number flows, see \cite{GS22a}.

The standard no-slip boundary condition $u=0$ corresponds to formally setting $\alpha = \infty$ above. In this case, global-in-time existence of a weak Leray-Hopf solution, for initial data $u_0$ divergence-free, tangent to the boundary and square-integrable, is due to J. Leray, see \cite{L34a}. Additionally, the exponential decay of the solution in $L^2$ follows from the energy inequality, by using the Poincar\'e inequality.

In the case $0< \alpha < \infty$ the literature contains no complete proof, aside from some special cases, of existence of weak (Leray-Hopf) solutions or discussion of large time behavior. Global existence of weak solutions was proved in the case $\alpha = 0$ in \cites{FG2022DCDS,Watanabe2003}. A sketch of a proof for $0 \leq \alpha < \infty$ was given in \cite{IS2011}. The rigorous analysis for global existence was presented for several boundaries conditions in \cite{BoyerFabrie2013} but the Navier boundary condition is not covered. In \cites{AACG2021,AEG2021} local well-posedness of strong solutions was established using semigroup methods and mild solutions, i.e. in which the equations are written in integral form, under appropriate regularity hypotheses on $\alpha$ and assuming the initial velocity is in $L^r$, $r \geq 3$. Global existence, under these conditions, was obtained for small data. For two-dimensional domains, Clopeau, Mikeli{\'c} and Robert proved global well-posedness of weak solutions and convergence of the vanishing viscosity limit in \cite{CMR}. Their analysis is based on the vorticity $\omega:=\curl u$ and the stream function $\psi:=\Delta^{-1}_{0}\omega$ (solution of the Laplace problem with the Dirichlet boundary condition) and hence applied only to 2D, simply connected, smooth, bounded domains. Taking into account harmonic vector fields and circulations, the first author extended their result to domains with holes in \cite{KNavier}. The goal here is to present a complete and self-contained analysis for 3D bounded domains.

 More precisely, let $L^2_{\sigma,\tang}(\Omega)$ be the space of square-integrable solenoidal vector fields on $\Omega$ which are tangent to $\partial \Omega$. We prove existence of a weak solution to $\eqref{NSwNBC}$, assuming that the initial velocity $u_0 \in L^2_{\sigma,\tang}(\Omega)$, and that $\alpha \in C^0(\partial\Omega)$ is a time-independent friction coefficient. The precise notion of weak solution is one of the issues that must be addressed, and will be discussed later. In addition, we establish exponential decay of the solution in energy norm in two scenarios:
\begin{itemize}
\item The friction coefficient $\alpha$ is {\it strictly positive}.
\item The friction coefficient $\alpha$ is {\it non-negative} and the domain $\Omega$ is such that $\Ker(S) = \{0\}$, as an operator on $L^2_{\sigma,\tang}(\Omega)$.
\end{itemize}
Furthermore, if the friction coefficient vanishes identically, we prove that the weak solution decays exponentially to the projection of the initial data onto the kernel of the operator $S$. It is a well-known fact, which we prove for the sake of completeness, that $\Ker(S)$ consists of vector fields which generate a rigid rotation of the domain.

In broad terms, we follow the general strategy developed for similar results when $\alpha = \infty$. Existence is tackled by passing to the limit in a Galerkin approximation and the exponential decay is obtained by combining a Poincar\'{e}-type inequality with an integral Gronwall-type inequality. However, the Navier friction condition introduces a few technical hurdles that must be addressed. Specifically, we must account for the various equivalent weak formulations of the problem, the precise choice of basis for the Galerkin approximation, in a way that is consistent with the Navier condition, and the influence of the geometry of the boundary.

In \cite{BJ21}, H. Baba and M. Jazar studied the long-time behavior of a weak solution to the Navier-Stokes system in a bounded 3D domain with Navier-type boundary conditions. In their work, the friction condition $[(2Su)\n + \alpha u]_{\tang} = 0$ is replaced by assuming vorticity is normal to the boundary. These types of boundary conditions are only equivalent in special cases, namely, if $\Omega$ is a half-space, and $\alpha = 0$, or if $\Omega$ is a sphere of radius $r>0$ and $\alpha = 2/r$. Moreover, their result only applies to initial data which satisfies their boundary condition and lies in the orthogonal complement of the kernel of the Stokes operator within that initial data space. Furthermore, they only obtain algebraic decay of the solution in energy norm. Recently a study on the long-time asymptotics of strong solutions with $\alpha = 0$ has become available, see \cite{BFP24}, which rediscovers some of the details in the present work.

Additional related work on Navier-Stokes and Stokes with Navier-boundary conditions include \cites{SSc1973,Berselli2010,BdVeiga2004} and references therein.

 One important aspect of the literature concerning the Navier friction condition is the vanishing viscosity limit in domains with boundary. This is a classical open problem in the no-slip case and we refer the reader to \cite{BNNT2022,LN2018} for recent accounts of the state-of-the art. In contrast, the corresponding problem with Navier friction boundary conditions is much more treatable, basically because the associated boundary layer is less singular. There is a large literature connected with this problem, beginning with the work of Clopeau {\it et al} on the vanishing viscosity limit in two-dimensional bounded domains, see \cite{CMR}. Concerning three-dimensional bounded domains, there is one result of particular interest here. In \cite{IP2006}, Iftimie and Planas proved convergence of the vanishing viscosity limit for Leray-type weak solutions of \eqref{NSwNBC} in energy norm, up to the time of existence of a strong solution of the Euler equations. The authors of \cite{IP2006} omit the proof of global existence of these Leray-type weak solutions, thus the existence result in the present work completes their analysis. The connection with the geometry of the domain, a central concern in our work, is not discussed in \cite{IP2006}.

The remainder of this work is organized as follows. In Section~\ref{sec2} we derive a weak formulation for problem $\eqref{NSwNBC}$. In Section~\ref{StokesOpSection} we study the Stokes operator with Navier boundary conditions and we prove a spectral theorem for this operator. Section~\ref{sec4} contains the statement and proof of the existence result, including two formulations of the energy inequality. In Section~\ref{sec5} we prove a Poincar\'{e}-type inequality in $H^1_{\sigma,\tang}$, in terms of the symmetric Jacobian of the velocity. In Section~\ref{sec6} we establish exponential decay of weak solutions in the case where $\alpha \geq 0$ and $\Ker(S) = \{0\}$ and in the case $\alpha > 0$ with no further restrictions on the domain. In Section~\ref{sec7}, we prove that velocity fields in $\Ker(S)$ are steady solutions of $\eqref{NSwNBC}$ and that weak solutions with initial data $u_0$ decay exponentially in time to the projection of $u_0$ onto $\Ker(S)$. We collect final remarks and conclusions in Section~\ref{sec8}. Lastly, in an Appendix we prove the integral Gronwall-type inequality applicable to the generalized energy inequalities deduced in Section~\ref{sec4}. This result is key to derive the exponential decay we claimed.

\section{Weak formulation}\label{sec2}

In this section, we derive a weak formulation for \eqref{NSwNBC}. We begin by introducing notation for the function spaces we will be using. As mentioned in Section~\ref{Int}, we denote by $L^2_{\sigma,\tang}(\Omega)$ the space of square-integrable solenoidal vector fields on $\Omega$ which are tangent to $\partial \Omega$ (note that square-integrable, solenoidal fields have a well-defined trace of the normal component in $H^{-1/2}(\partial \Omega)$, see, for example, \cite[Equation (IV.10)]{BoyerFabrie2013}). The space $H^1_{\sigma,\tang}(\Omega)$ corresponds to the vector fields in $L^2_{\sigma,\tang}(\Omega)$ whose first-order weak derivatives are, additionally, square-integrable. The notation $C^\infty_{\sigma,\tang}(\overline{\Omega})$ refers to the $C^\infty$ solenoidal vector fields which are tangent to $\partial \Omega$. Finally, we use $C([0,T);w-L^2_{\sigma,\tang}(\Omega))$ to denote functions which are continuous from $[0,T)$ into $L^2_{\sigma,\tang}(\Omega)$ with the weak topology; $C_{\lo}([0,+\infty);w-L^2_{\sigma,\tang}(\Omega))$ denotes functions continuous from $[0,+\infty)$ into $L^2_{\sigma,\tang}(\Omega)$ with the weak topology, and which belong to $L^\infty([0,T);L^2)$ for every $T>0$.

We will assume, throughout this paper, that $\alpha \in C^0(\partial\Omega)$ is a nonnegative time-independent friction coefficient.

We require some elementary information on differential geometry of surfaces, which we recall below.

For each $p \in \partial \Omega$ let $\n=\n_p$ be the outward unit normal to $\partial \Omega$ at $p$. This induces a map $\n: \partial \Omega \to S^2$, where $S^2$ is the unit sphere in $\R^3$, called the {\it Gauss map}. Its differential is $d\n_p: T_p(\partial \Omega) \to T_p (\partial \Omega)$, using the natural identification $T_{\n(p)}(S^2) \sim T_p (\partial \Omega)$. The map $d\n_p $, the differential of the Gauss map, is called the {\em shape operator} of $\partial \Omega$ and, for each $p$, it is a self-adjoint linear operator. See \cite[Section~3.2, Proposition 1]{Manfredo2016} for details. The eigenvalues of $-d\n_p$ are the {\em principal curvatures} of $\partial\Omega$ at $p$, denoted $k_1=k_1(p)$ and $k_2=k_2(p)$. Let $\lambda$ be defined as follows:
\begin{equation} \label{lambdaofp}
\lambda = \lambda (p) := \max\{|k_1(p)|,|k_2(p)|\}.
\end{equation}

To obtain a weak formulation of \eqref{NSwNBC} first assume that $u$ is a smooth solution of the system and let $\Phi \in C^\infty_c([0,+\infty);C^\infty_{\sigma, \tang}(\overline{\Omega}))$. Taking the inner product of \eqref{NSwNBC} with $\Phi$ and integrating by parts once yields:
\begin{multline} \label{motivwkform}
\int_0^{+\infty} \int_\Omega \{ \partial_t \Phi \cdot u \, + \, [(u \cdot \nabla )\Phi]\cdot u \} \dd x \dd t + \int_\Omega \Phi(0,x)\cdot u_0(x) \dd x \\
 = \nu \left(\int_0^{+\infty}\int_\Omega D\Phi : Du \dd x \dd t - \int_0^{+\infty}\int_{\partial\Omega} \Phi \cdot (Du \, \n) \dd S \dd t\right).
\end{multline}
Above, the notation $A:B$ stands for the trace of the matrix product $AB$ and $\dd S$ is the $2$-dimensional Hausdorff measure on $\partial \Omega$. We will now examine the boundary integral more carefully, using the Navier boundary condition and the differential of the Gauss map.

\begin{lemma} \label{bdrytermNBC}
If $\Phi\in C^\infty_{\sigma, \tang}(\overline{\Omega})$ and if $u \in C^\infty_{\sigma, \tang}(\overline{\Omega})$ satisfies $[(2Su)\n + \alpha u]_{\tang} = 0$ on $\partial \Omega$ then it holds that, on $\partial \Omega$,
\[ \Phi \cdot (Du \, \n) = \Phi \cdot [d\n(u) - \alpha u].
\]

\end{lemma}

\begin{proof}
Fix $p\in\partial\Omega$. Let $\tau \in T_p(\partial\Omega)$ and choose a curve $\gamma: q=q(s)$ on $\partial\Omega$ such that $q(0)=p$ and $\dot{q}(0)=\tau$. Since we will be working with the outward unit normal vector to $\partial\Omega$ at $q$, for different points $q \in \partial\Omega$, it is convenient, in this proof, to make the dependence of $\n$ on $q$ explicit: $\n=\n_q$.

Since $u\cdot\n_q = 0$ for all $q$ on the curve $\gamma$, and using the definition of the differential of the Gauss map, we find
\[0=\frac{\mathrm{d}}{\mathrm{d}s}(u\cdot\n_q)= (Du \,\dot{q}(s))\cdot \n_q + u \cdot d\n_q(\dot{q}(s)).\]
Recalling that $d\n_q$ is self-adjoint and that $u$ is tangent to $\partial\Omega$, {\em i.e.} $u \in T_q(\partial\Omega)$, we deduce that
\begin{equation} \label{Dutn}
((Du)^t \,\n_q )\cdot \dot{q}(s) = - d\n_q(u)\cdot \dot{q}(s).
\end{equation}
Next we add $(Du \,\n_q )\cdot \dot{q}(s)$ to both sides above and, since $2Su = Du + (Du)^t$, we obtain that
\[ (2Su \,\n_q )\cdot \dot{q}(s) = (Du \,\n_q )\cdot \dot{q}(s)- d\n_q(u)\cdot \dot{q}(s).\]
In particular, at $q(0)=p$, we find
\[ (2Su \,\n_p )\cdot \tau = (Du \,\n_p )\cdot \tau - d\n_p(u)\cdot \tau,\]
and hence
\[ (Du \,\n_p )\cdot \tau = [d\n_p(u) + (2Su \,\n_p )]\cdot \tau .\]

Recall that $\tau$ was chosen as an arbitrary vector in $T_p(\partial\Omega)$. We use the Navier boundary condition satisfied by $u$ to deduce that
\[ (Du \,\n_p )_{\tang} = [ d\n_p(u) - \alpha u ]_{\tang} .\]

This identity is valid at any point $p$ on $\partial\Omega$, so we can now abandon the explicit mention to $p$ in $d\n_p$.
Since the normal component of $\Phi$ vanishes on $\partial\Omega$ it follows that
\[ \Phi \cdot (Du \,\n )= \Phi_{\tang} (Du \,\n )_\tang = \Phi_\tang [ d\n (u) - \alpha u ]_\tang = \Phi \cdot [ d\n (u) - \alpha u ],\]
as we wished.

\end{proof}

In what follows we use $(H^1_{\sigma,\tang}(\Omega))^\prime$ to denote the abstract dual space to $H^1_{\sigma,\tang}(\Omega)$.

We are now ready to introduce the definition of a weak solution to \eqref{NSwNBC}.

\begin{definition} \label{WeakFormWithDu}
Fix $\alpha \in C^0(\partial\Omega)$ and $u_0 \in L^2_{\sigma,\tang}(\Omega)$. Let $u \in C_{\lo}([0,+\infty);w-L^2_{\sigma,\tang}(\Omega)) \cap L^2_\lo((0,+\infty);H^1_{\sigma,\tang}(\Omega))$. In addition, suppose that $\partial_t u \in L^{4/3}_\lo((0,+\infty);(H^1_{\sigma,\tang}(\Omega))^\prime)$. We say $u$ is a {\em weak solution} of \eqref{NSwNBC} with initial data $u_0$ if, for every test vector field $\Phi \in C^\infty_c([0,+\infty);C^\infty_{\sigma, \tang}(\overline{\Omega}))$, it holds that
\begin{multline} \label{WeakFormWithDuIdentity}
\int_0^{+\infty} \int_\Omega \left\{\partial_t\Phi\cdot u\, + \, [(u \cdot \nabla )\Phi]\cdot u \right\}\dd x \dd t + \int_\Omega \Phi(0,x)\cdot u_0(x) \dd x \\
 = \nu\int_0^{+\infty}\left( \int_\Omega D\Phi : Du \dd x + \int_{\partial\Omega} \Phi \cdot [\alpha u - d\n (u)] \dd S\right)\dd t.
\end{multline}

\end{definition}

\section{The Stokes operator} \label{StokesOpSection}

We denote by $\mathbb{P}$ the orthogonal projection of the space of square-integrable vector fields in $\Omega$ onto $L^2_{\sigma,\tang}(\Omega)$, also known as the {\it Leray projector}. To eliminate the pressure from the Navier-Stokes equations, one applies $\mathbb{P}$ to it. Clearly, we have $\mathbb{P}\partial_t u= \partial_t\mathbb{P} u = \partial_t u$, as $u$ is already divergence-free and tangent to the boundary. The nonlinear term becomes $\mathbb{P}[u \cdot \nabla u]$ and the viscous term becomes $\nu \mathbb{P} \Delta u$. As is well-known, the Laplacian does not commute with the Leray projector, despite the fact that $\Delta u$ is divergence free, because the vector field $\Delta u$ is not, in general, tangent to $\partial \Omega$.

The elliptic operator $-\mathbb{P} \Delta$ is called the {\it Stokes operator} and the semigroup it generates is called the {\it Stokes semigroup}. These objects play a key role in the analysis of the viscous flow equations. For flows with no-slip boundary conditions, the natural phase space is $H^1_{\sigma,0}(\Omega)$, and the properties of the Stokes operator and associated semigroup acting on this space are well-understood. For the present work, we require properties of the Stokes operator, when acting on the space of vector fields satisfying the Navier boundary conditions. This is the subject of the present section.

We return to the calculation performed to obtain \eqref{motivwkform} and we concentrate on the Laplacian term. Since we are only concerned with spatial derivatives, we consider time-independent vector fields. More precisely, assume $u$ is a time-independent, smooth, divergence-free vector field on $\Omega$, tangent to $\partial \Omega$, and let $\Phi \in C^\infty_{\sigma,\tang}(\overline{\Omega})$. Integration by parts yields the identity
\begin{equation*} 
 - \int_\Omega \Phi \, \Delta u = \int_\Omega D\Phi : Du
 - \int_{\partial\Omega} \Phi \cdot (Du \, \n) \dd S .
\end{equation*}
Assume, further, that $u$ satisfies the Navier boundary condition $[(2Su)\n + \alpha u]_{\tang} = 0$. Then, using the result in Lemma~\ref{bdrytermNBC} we obtain
\begin{align} \label{Stokes1}
 - \int_\Omega \Phi \, \Delta u = \int_\Omega & D\Phi : Du + \int_{\partial\Omega} \Phi \cdot [\alpha u - d\n(u) ] \dd S .
\end{align}
We see that the right-hand-side above is well-defined for $u \in H^1_{\sigma,\tang}(\Omega)$ and $\Phi \in H^1_{\sigma,\tang}(\Omega)$. We thus introduce the Stokes operator, acting on flows satisfying the Navier boundary conditions, as
\[
\begin{array}{ccc}\mathbb{A}: H^1_{\sigma,\tang}(\Omega) & \to & (H^1_{\sigma,\tang}(\Omega))^\prime \\
 u & \mapsto & \mathbb{A}u
\end{array},
\]
where $\mathbb{A}{u}$ is defined through the duality relation
\begin{equation}\label{StokesOp}
 \langle \mathbb{A}u , v \rangle := \int_\Omega Dv : Du + \int_{\partial\Omega} v \cdot [\alpha u - d\n(u)] \dd S , \quad v \in H^1_{\sigma,\tang}(\Omega).
\end{equation}

\begin{lemma} \label{stokeslemma} The right-hand-side of \eqref{StokesOp} gives rise to a bounded bilinear symmetric operator $B: H^1_{\sigma,\tang}(\Omega) \times H^1_{\sigma,\tang}(\Omega) \to \mathbb{R}$.
\end{lemma}

\begin{proof}
Let $(u,v) \in H^1_{\sigma,\tang}(\Omega) \times H^1_{\sigma,\tang}(\Omega)$ and define $B: H^1_{\sigma,\tang}(\Omega) \times H^1_{\sigma,\tang}(\Omega) \to \mathbb{R}$ by
\[B(u,v) = \int_\Omega Dv : Du + \int_{\partial\Omega} v \cdot [\alpha u - d\n(u)] \dd S.\]

Recall $\lambda$ as introduced in \eqref{lambdaofp}. Then
\begin{align*}
|B(u,v)| & = \left| \int_\Omega Dv : Du + \int_{\partial\Omega} v \cdot [\alpha u - d\n(u)] \dd S \right| \\
& \leq C \|u\|_{H^1} \|v \|_{H^1} + \left(\|\lambda\|_{L^\infty(\partial\Omega)} + \|\alpha\|_{L^\infty(\partial\Omega)}\right)\|u\|_{L^2(\partial\Omega)}\|v\|_{L^2(\partial\Omega)} \\
& \leq (C + \|\lambda\|_{L^\infty(\partial\Omega)} + \|\alpha\|_{L^\infty(\partial\Omega)}) \|u\|_{H^1} \|v \|_{H^1},
\end{align*}
where the last inequality is a consequence of the continuity of the trace of functions in $H^1(\Omega)$ onto $L^2(\partial\Omega)$.
\end{proof}

 Recall the trace inequality (see, for instance, \cite[Theorem III.2.19]{BoyerFabrie2013}),
 \begin{equation}\label{traceineq}
 \|u\|_{L^2(\partial\Omega)}^2 \lesssim \|u \|_{L^2(\Omega)}\|u \|_{H^1(\Omega)}.
 \end{equation}
 Note, also, that a divergence-free vector field $u$ on $\Omega$, which is tangent to $\partial\Omega$, has mean zero for each of its components $u_j$, $j=1,$ $2,$ $3$. This follows easily by integrating by parts $\ds{\int_{\Omega} u \cdot \nabla x_j}$. Recall the Poincar\'e inequality for mean-free functions
 \begin{equation} \label{meanfreePoinc}
 \|u \|_{H^1(\Omega)} \leq c\|Du \|_{L^2(\Omega)}.
 \end{equation}

We claim that there exists $\beta >0$ such that $\widetilde{B}:= B + \beta \mathbb{I}$, with $\mathbb{I}$ denoting the identity operator, is
 {\it coercive} with respect to the $H^1$-norm, {\it i.e.}, there exists $K>0$ such that $\widetilde{B}(u,u) \geq K\|u\|_{H^1}^2$ for all $u \in H^1_{\sigma,\tang}(\Omega)$. Indeed, this follows immediately from the estimate below:
 \begin{align} \label{Buu}
 B(u,u) & = \int_\Omega |Du|^2 + \int_{\partial\Omega} (\alpha |u|^2 - d\n (u)\cdot u) \dd S \nonumber \\
 & \geq \|Du\|_{L^2}^2 - (\|\alpha\|_{L^\infty(\partial\Omega)} + \|\lambda\|_{L^\infty(\partial\Omega)})\|u\|_{L^2(\partial\Omega)}^2 \nonumber \\
 & \geq \|Du\|_{L^2}^2 - (\|\alpha\|_{L^\infty(\partial\Omega)} + \|\lambda\|_{L^\infty(\partial\Omega)})C\|u\|_{L^2(\Omega)}\|Du\|_{L^2(\Omega)} \nonumber \\
 & \geq \frac{1}{2}\|Du\|_{L^2}^2 - \beta\|u\|_{L^2(\Omega)}^2 \nonumber \\
 & \geq K\|u\|_{H^1}^2 - \beta\|u\|_{L^2(\Omega)}^2.
 \end{align}
 Above, in the second inequality we used the trace inequality \eqref{traceineq}. In the third inequality we used Young's inequality with
 \[\beta=\frac{(\|\alpha\|_{L^\infty(\partial\Omega)} + \|\lambda\|_{L^\infty(\partial\Omega)})^2 C^2}{2}.\]
 Finally, in the fourth inequality we used the Poincar\'e inequality \eqref{meanfreePoinc}, where $K=(2c^2)^{-1}$.

 It follows from \eqref{Buu} that $\widetilde{B}$ is a positive-definite bilinear operator on $H^1_{\sigma,\tang}(\Omega)$, thus an inner product. We use the notation
 \begin{equation}\label{tilBinnprod}
 ((u,v)):= \widetilde{B}(u,v)
 \end{equation}
 for this inner product and we note that it gives rise to an equivalent norm on $H^1_{\sigma,\tang}(\Omega)$.

\begin{prop} \label{eigenfunctionsStokes}
There exists a sequence
	$\{v_j\}_{j=1}^\infty \subset H^1_{\sigma,\tang}(\Omega)$ of eigenfunctions of $\mathbb{A}$, with corresponding eigenvalues $\{\lambda_j\}_{j=1}^\infty \subset \R$, which form an orthonormal basis
	of $L^2_{\sigma,\tang}(\Omega)$ and, also, a basis of $H^1_{\sigma,\tang}(\Omega)$ which is orthogonal with respect to the inner product $((\cdot,\cdot))$.
	The eigenvalues, ordered increasingly, satisfy $\lambda_j \to \infty$ as $j \to \infty.$
\end{prop}

\begin{proof}
 Let us introduce $\widetilde{\mathbb{A}} := \mathbb{A} + \beta\mathbb{I}$, where $\beta$ is the constant from estimate \eqref{Buu}.
 Consider $\text{Dom}(\widetilde{\mathbb{A}}) := \{ u \in H^1_{\sigma,\tang}(\Omega)\,|\,\widetilde{\mathbb{A}}u \in L^2_{\sigma,\tang}(\Omega)\}$ and note that $\text{Dom}(\widetilde{\mathbb{A}})$ is compactly imbedded in $L^2_{\sigma,\tang}(\Omega)$. Furthermore, it is easy to see that $\text{Dom}(\widetilde{\mathbb{A}})$ is dense in $L^2_{\sigma,\tang}(\Omega)$ with respect to the $L^2$-norm.

 The coerciveness of $\widetilde{B}$, together with the continuity and symmetry of $B$ and the compactness mentioned above, imply that the operator $\widetilde{\mathbb{A}}$ has an inverse which is a self-adjoint compact operator from $L^2_{\sigma,\tang}(\Omega)$ to $\text{Dom}(\widetilde{\mathbb{A}}) \subset H^1_{\sigma,\tang}(\Omega)$. Therefore it follows from the spectral theory for compact self-adjoint operators that there exists an orthonormal basis of $L^2_{\sigma,\tang}(\Omega)$ of eigenvectors of $\widetilde{\mathbb{A}}$, along with an increasing sequence of eigenvalues tending to $+\infty$, see \cite[Section 6, Chapter II]{BoyerFabrie2013}. This basis is also orthogonal with respect to the $((\cdot,\cdot))$ inner product on $H^1_{\sigma,\tang}(\Omega)$, and it is a complete set, \cite[Corollary II.3.8 and Theorem IV.5.5]{BoyerFabrie2013}. Clearly, an eigenvector $v_j$ of $\widetilde{\mathbb{A}}$ is also an eigenvector of $\mathbb{A}$, albeit with eigenvalue $\lambda_j - \beta$.
 \end{proof}

\section{Existence of weak solution}\label{sec4}

In this section we establish global existence of a weak solution to \eqref{NSwNBC} with initial data in $L^2_{\sigma,\tang}$, using a Galerkin approximation, and we derive two formulations of the energy inequality.

\begin{theorem} \label{ExistWeakSolWithDu}
Let $\alpha \in C^0(\partial \Omega)$ and $u_0 \in L^2_{\sigma,\tang}(\Omega)$ be given. Then there exists $u \in C_{\loc}([0,+\infty);w-L^2_{\sigma,\tang}(\Omega)) \cap L^2_{\loc}((0,+\infty);H^1_{\sigma,\tang}(\Omega))$, such that, $\partial_t u \in L^{4/3}_{\loc}((0,+\infty);(H^1_{\sigma,\tang}(\Omega))^\prime)$ and, for any
test vector field $\Phi \in C^\infty_c([0,+\infty);C^\infty_{\sigma,\tang}(\overline{\Omega}))$, the identity \eqref{WeakFormWithDuIdentity} holds true.
\end{theorem}

\begin{remark} \label{existencewksolDu}
Theorem~\ref{ExistWeakSolWithDu} is equivalent to the existence of a weak solution in the sense of Definition~\ref{WeakFormWithDu}.
\end{remark}

\begin{proof}
 The proof is a slight adaptation of the well-known proof of existence of weak Leray solutions. We outline the main steps and highlight only the differences with respect to the standard argument.

 The first step is to consider the orthonormal basis of $L^2_{\sigma,\tang}(\Omega)$ given by the eigenfunctions of $\mathbb{A}$, $\{v_j\}_j \subset H^1_{\sigma,\tang}(\Omega)$, obtained in Proposition~\ref{eigenfunctionsStokes}. We use these vector fields to build finite dimensional approximations $u^M$ through the Galerkin method, by projecting the (weak form of the) PDE onto the space generated by the first $M$ vector fields in the basis, namely $X^M := $ span$\{ v_1, \ldots, v_M \}$. The approximations $u^M$ are given by $u^M = \sum_{j = 1}^M g_{j}^M v_j$, with coefficients $g_j^M = g_j^M(t)$ which are solutions of a system of quadratic ODEs with constant coefficients. Short-time existence is an easy consequence of Picard's theorem. The resulting $u^M$ satisfy
 \begin{align} \label{GalerkinWithDu}
 \int_{\Omega} &\left\{ (\partial_t u^M) \cdot v_j + [(u^M \cdot \nabla) u^M]\cdot v_j \right\}\dd x \nonumber \\
 & = -\nu \left(\int_\Omega Du^M : D v_j \dd x + \int_{\partial\Omega} [\alpha u^M -d\n(u^M)]\cdot v_j \dd S \right).
 \end{align}
 Standard energy estimates for \eqref{GalerkinWithDu} give:
 \begin{equation*}
 \frac{d}{dt}\|u^M\|_{L^2(\Omega)}^2 \leq -2\nu \|Du^M\|_{L^2(\Omega)}^2 + 2\nu\left(\|\lambda\|_{L^\infty(\partial\Omega)} + \|\alpha\|_{L^\infty(\partial\Omega)}\right)\|u^M\|_{L^2(\partial\Omega)}^2,
 \end{equation*}
 where $\lambda$ was introduced in \eqref{lambdaofp}. We use the trace inequality \eqref{traceineq} for $u^M$, together with the Poincar\'e inequality \eqref{meanfreePoinc}, since the mean of $u^M$ is zero, so that $\|u^M\|_{H^1(\Omega)} \lesssim \|Du^M\|_{L^2(\Omega)}$, followed by Young's inequality, to deduce that
 \begin{equation}\label{EnEstimateForUM2}
 \frac{d}{dt}\|u^M\|_{L^2(\Omega)}^2 + \nu \|Du^M\|_{L^2(\Omega)}^2 \leq C \nu\left(\|\lambda\|_{L^\infty(\partial\Omega)} + \|\alpha\|_{L^\infty(\partial\Omega)}\right)\|u^M\|_{L^2(\Omega)}^2.
 \end{equation}
 The Gr\"onwall lemma yields $\{u^M\}$ bounded in $L^\infty_\loc((0,+\infty);L^2_{\sigma,\tang}(\Omega))\cap L^2_\loc((0,+\infty);H^1_{\sigma,\tang}(\Omega))$, from which we deduce global existence of $u^M$ for fixed $M$. Furthermore, the bounds on $u^M$ are uniform with respect to $M$.

 In order to obtain, from these bounds, the boundedness of $\{\partial_t u^M\}$ in $L^{4/3}_\loc((0,+\infty);(H^1_{\sigma,\tang}(\Omega))^\prime)$, uniformly in $M$, we need to use the orthogonality of $\{v_j\}$ with respect to the inner product $((\cdot,\cdot))$ introduced in \eqref{tilBinnprod}. Let $\mathbb{P}^M$ denote the $L^2$-orthogonal projection onto $X^M$ and note that, for $t>0$ fixed, both $\partial_t u^M$ and $\mathbb{A}u^M$ belong to $X^M$, since $ \mathbb{A} v_j=\lambda_j v_j\in X^M$. Using the notation from Section~\ref{StokesOpSection} we note that the right-hand-side of \eqref{GalerkinWithDu} is a weak formulation of the right-hand-side of:
	\begin{align} \label{strongGalerkin}
		\begin{cases}
			 \, \partial_t u^M + \mathbb{P}^M [(u^M \cdot \nabla)u^M]=
- \nu \mathbb{A} u^M & \text{ in } \Omega,\\
\, \dv u^M = 0 & \text{ in } \Omega, \\
			\, u^M\cdot \n = 0, \quad
[(2Su^M)\n + \alpha u^M]_{\tang} = 0 & \text{ on } \partial \Omega.
		\end{cases}
	\end{align}

Let $W \in H^1_{\sigma,\tang}(\Omega)$. Then $W = \mathbb{P}^M W + Y$, where $Y$ is orthogonal to $\mathbb{P}^M W$ with respect to {\em both} the $L^2$ inner product and the $((\cdot,\cdot))$ inner product. In particular, using the $((\cdot,\cdot))$ inner product and its induced, $H^1$-equivalent, norm, it holds that
\begin{equation}\label{Pythagoras}
 \|\mathbb{P}^M W\|_{H^1} \leq C\|W\|_{H^1}.
\end{equation}
Then, using the definition of $\mathbb{A}$ by duality, \eqref{StokesOp}, we have
\begin{align*}
 |\langle \partial_t u^M , W \rangle| & = |\langle \partial_t u^M , \mathbb{P}^M W \rangle| \\
 & \leq \left| \int_\Omega [(u^M \cdot \nabla)u^M] \cdot \mathbb{P}^M W \dd x \right| \\
 & + \nu \left| \int_\Omega Du^M : D \mathbb{P}^M W \dd x \right| \\
 & +\nu \left| \int_{\partial\Omega} [\alpha u^M -d\n(u^M)]\cdot \mathbb{P}^M W \dd S \right|.
\end{align*}

Each of the three integrals above may be estimated in a standard way, using H\"{o}lder's inequality, interpolation, and the Sobolev imbedding $H^1 \subset L^6$ for the first term, Cauchy-Schwartz for the second, and the trace theorem for the boundary term, to find:
\begin{equation*}
 |\langle \partial_t u^M , W \rangle| \leq C\left(\|u^M\|_{L^2}^{1/2}\|u^M\|_{H^1}^{3/2} +\|u^M\|_{H^1}\right)\|\mathbb{P}^M W\|_{H^1}.
\end{equation*}
It then follows from \eqref{Pythagoras} that
\begin{equation*}
\|\partial_t u^M \|_{(H^1_{\sigma,\tang})^\prime} \leq C\left(\|u^M\|_{L^2}^{1/2}\|u^M\|_{H^1}^{3/2} +\|u^M\|_{H^1}\right).
\end{equation*}
Now, recall that we showed, from the energy estimate \eqref{EnEstimateForUM2}, that $u^M$ is bounded in $L^\infty_{\loc}((0,+\infty);L^2_{\sigma,\tang}(\Omega))\cap L^2_{\loc}((0,+\infty);H^1_{\sigma,\tang}(\Omega))$. In view of this we obtain, immediately, the desired bound: $\partial_t u^M$ is bounded in $L^{4/3}_{\loc}((0,+\infty);(H^1_{\sigma,\tang}(\Omega))^\prime)$, uniformly with respect to $M$.

 To conclude the proof of existence, a standard compactness argument allows us to pass to the limit $M\to +\infty$. Indeed, we may choose a subsequence, not relabeled, such that $u^M$ converges weakly-$\ast$ in $L^\infty_{\loc}((0,+\infty);L^2_{\sigma,\tang}(\Omega))$, weakly in $L^2_{\loc}((0,+\infty);H^1_{\sigma,\tang}(\Omega))$, strongly in $L^2_{\loc}((0,+\infty);L^2_{\sigma,\tang}(\Omega))$ to a limit $u$. In particular this implies, by the trace inequality \eqref{traceineq}, that the trace of $u^M$ on $\partial\Omega$ converges strongly in $L^2_{\loc}((0,+\infty);L^2_{\sigma,\tang}(\partial\Omega))$ to the corresponding trace of $u$. Furthermore, the Aubin-Lions compactness theorem implies $u \in C_{\loc}([0,+\infty);w-L^2_{\sigma,\tang}(\Omega))$.

 Using $\Phi \in C^\infty_c([0,+\infty);C^\infty_{\sigma,\tang}(\overline{\Omega}))$ as test vector field in a weak formulation of \eqref{strongGalerkin}, we can pass to the limit $M\to \infty$ and obtain the weak formulation \eqref{WeakFormWithDuIdentity}. This concludes the proof.
 \end{proof}

Our next result concerns a different way of writing the term on the right hand side of \eqref{WeakFormWithDuIdentity}, under the time integral. This will be useful to write \eqref{WeakFormWithDuIdentity} in a different way, which will lead to two sets of energy inequalities.

\begin{prop} \label{EquivWeakFormStokesOp}
Let $u \in H^1_{\sigma,\tang}(\Omega)$. Then, for every $\Phi\in C^\infty_{\sigma,\tang}(\overline{\Omega})$ it holds that
\begin{align} \label{WeakFormStokesOpId}
\int_\Omega D\Phi : Du \dd x + \int_{\partial\Omega} \Phi \cdot [\alpha u - & d\n (u)] \dd S \nonumber \\
= & 2\int_\Omega S\Phi : Su \dd x + \int_{\partial\Omega} \alpha\Phi \cdot u \dd S.
\end{align}
\end{prop}

\begin{proof}
 First assume that $u$ is smooth, divergence free and tangent to $\partial \Omega$. Let $\Phi \in C^\infty_{\sigma,\tang}(\overline{\Omega})$. Recall expression \eqref{Dutn} deduced in the proof of Lemma~\ref{bdrytermNBC}. The same argument used there yields
 \[[(Du)^t \,\n]_\tang = - [d\n(u)]_\tang.\]

 We compute
 \begin{align*}
 \int_\Omega D\Phi : Du \dd x = & \int_\Omega D\Phi : 2 Su \dd x - \int_\Omega D\Phi : (Du)^t \dd x \\
 = & 2 \int_\Omega S\Phi : Su \dd x - \int_\Omega D\Phi : (Du)^t \dd x.
 \end{align*}

 Next we observe that, since $u$ is divergence free,
 \[ - \int_\Omega D\Phi : (Du)^t \dd x
 = - \int_{\partial \Omega} \Phi \cdot [(Du)^t\n] \dd S = \int_{\partial\Omega} \Phi \cdot [d\n(u)] \dd S.\]
 This establishes \eqref{WeakFormStokesOpId} for smooth $u$. We conclude the proof by density.
 \end{proof}

In view of Proposition~\ref{EquivWeakFormStokesOp} we have obtained an alternative weak formulation, equivalent to \eqref{WeakFormWithDuIdentity}, namely,
\begin{align} \label{WeakFormWithSuIdentity}
\int_0^{+\infty} \int_\Omega & \left\{\partial_t\Phi\cdot u\, + \, [(u \cdot \nabla )\Phi]\cdot u \right\}\dd x \dd t + \int_\Omega \Phi(0,x)\cdot u_0(x) \dd x \nonumber \\ & = \nu\int_0^{+\infty}\left( \int_\Omega 2S\Phi : Su \dd x +
\int_{\partial\Omega} \alpha \Phi \cdot u \dd S\right)\dd t.
\end{align}

We will now amend our existence theorem to include this new weak formulation and we note that two different energy inequalities arise.

\begin{theorem} \label{GlobalExistWeakSol}
Let $\alpha \in C^0(\partial \Omega)$ and $u_0 \in L^2_{\sigma,\tang}(\Omega)$ be given. Then there exists a weak solution $u$, according to Definition~\ref{WeakFormWithDu}, such that, for any test vector field $\Phi \in C^\infty_c([0,+\infty);C^\infty_{\sigma,\tang}(\overline{\Omega}))$, \eqref{WeakFormWithSuIdentity} holds, in addition to identity \eqref{WeakFormWithDuIdentity} holding true.

Moreover, for almost every $s \geq 0$ and for every $t \geq s$, $u$ satisfies both of the following energy inequalities:
\begin{align}
\label{EnInDu}
\|u(t ) \|_{L^2}^2 + & 2\nu \int_{s}^{t } \|Du\|_{L^2}^2 \dd \tau \nonumber \\
& \leq \|u(s)\|_{L^2}^2 + 2\nu \int_{s}^{t } \int_{\partial\Omega} \left(\|\lambda\|_{L^\infty(\partial\Omega)} - \alpha\right) |u|^2 \dd S \dd \tau ,
\\
\text{and } \nonumber
\\
\label{EnInSu}
\|u(t ) \|_{L^2}^2 + & 4\nu \int_{s}^{t } \|Su\|_{L^2}^2 \dd \tau \leq \|u(s)\|_{L^2}^2 - 2\nu \int_{s}^{t } \int_{\partial\Omega} \alpha |u|^2 \dd S \dd \tau.
\end{align}
Finally, both \eqref{EnInDu} and \eqref{EnInSu} are satisfied for $s=0$.
\end{theorem}

\begin{proof}
The existence part is an easy consequence of the existence result, Theorem~\ref{ExistWeakSolWithDu}, together with the equivalence between the two weak formulations \eqref{WeakFormWithDuIdentity} and \eqref{WeakFormWithSuIdentity}. All that remains to prove are the energy inequalities.

To show \eqref{EnInDu} we revisit \eqref{GalerkinWithDu}, multiplying by $g^M_j$ and summing in $j$, bounding only the shape operator, to find
 \begin{equation} \label{blaDu1}
 \frac{d}{dt}\|u^M\|_{L^2(\Omega)}^2 \leq -2\nu \|Du^M\|_{L^2(\Omega)}^2 + 2\nu\int_{\partial\Omega}\left( \|\lambda\|_{L^\infty(\partial\Omega)} - \alpha\right) |u^M|^2 \dd S.
 \end{equation}
Let $r \geq s$ and integrate in time from $s$ to $r$ to obtain
 \begin{align} \label{blaDu2}
 \|u^M(r, \cdot)&\|_{L^2(\Omega)}^2 \leq -2\nu \int_{s}^r \|Du^M(\tau,\cdot)\|_{L^2(\Omega)}^2 \dd \tau \nonumber \\
 & + 2\nu \int_{s}^r \int_{\partial\Omega}\left( \|\lambda\|_{L^\infty(\partial\Omega)} - \alpha\right) |u^M|^2 \dd S \dd \tau + \|u^M(s,\cdot)\|_{L^2(\Omega)}^2 .
 \end{align}
Fix $t \geq s$. As is done in the case of no-slip boundary conditions, see \cite[Proposition 14.1]{LMRbook}, choose $\phi \in C^\infty_c(\real_+)$, $\int_{\real_+} |\phi|^2 dr = 1$, $\supp \phi \subset (0,1)$. Let $\vare > 0$ and set
$\phi_\vare = \phi_\vare(r)\equiv \frac{1}{\sqrt{\vare}}\phi\left(\frac{r-t }{\vare}\right)$, $r \in \real_+$. We multiply \eqref{blaDu2} by $|\phi_\vare(r)|^2$, move the term with $Du^M$ to the left-hand-side, and observe that $\phi_\vare$ vanishes if $r \leq s$. We may therefore integrate on $\real_+$ to deduce that
 \begin{align} \label{blaDu3}
 & \int_0^{+\infty} |\phi_\vare|^2(r) \|u^M(r, \cdot)\|_{L^2(\Omega)}^2 \dd t + 2\nu \int_0^{+\infty} |\phi_\vare|^2(r) \int_{s}^r \|Du^M(\tau,\cdot)\|_{L^2(\Omega)}^2 \dd \tau \dd r\nonumber \\
 & \leq 2\nu \int_0^{+\infty} |\phi_\vare|^2(r) \int_{s}^r \int_{\partial\Omega}\left( \|\lambda\|_{L^\infty(\partial\Omega)} - \alpha\right) |u^M|^2 \dd S \dd \tau \dd r + \|u^M(s,\cdot)\|_{L^2(\Omega)}^2 .
 \end{align}
Recall, from the proof of Theorem~\ref{ExistWeakSolWithDu}, that, passing to subsequences as needed, we may assume $u^M$ converges to a weak solution $u$ weakly in $L^2_{\loc}((0,+\infty);H^1_{\sigma,\tang}(\Omega))$ and strongly in $L^2_{\loc}((0,+\infty);L^2_{\sigma,\tang}(\Omega))$ and
 the trace of $u^M$ on $\partial\Omega$ converges strongly in $L^2_{\loc}((0,+\infty);L^2_{\sigma,\tang}(\partial\Omega))$ to the corresponding trace of $u$. Using these properties of convergence and passing to the $\liminf_{M \to +\infty}$ in the inequality \eqref{blaDu3} yields
 \begin{align*}
 & \int_0^{+\infty} |\phi_\vare|^2(r) \|u (r, \cdot)\|_{L^2(\Omega)}^2 \dd r + 2\nu \int_0^{+\infty} |\phi_\vare|^2(r) \int_{s}^r \|Du (\tau,\cdot)\|_{L^2(\Omega)}^2 \dd \tau \dd r\nonumber \\
 & \leq 2\nu \int_0^{+\infty} |\phi_\vare|^2(r ) \int_{s}^r \int_{\partial\Omega}\left( \|\lambda\|_{L^\infty(\partial\Omega)} - \alpha\right) |u|^2 \dd S \dd \tau \dd r \nonumber \\
 & +
 \liminf_{M \to + \infty} \|u^M(s,\cdot)\|_{L^2(\Omega)}^2 .
 \end{align*}
 Assume now that $t $ is a Lebesgue point of $r \mapsto \|u(r ,\cdot)\|_{L^2}^2$. Then, passing to the limit $\vare \to 0$ we find
 \begin{align} \label{blaDu5}
 & \|u (t , \cdot)\|_{L^2(\Omega)}^2 + 2\nu \int_{s}^{t } \|Du (\tau,\cdot)\|_{L^2(\Omega)}^2 \dd \tau \nonumber \\
 & \leq 2\nu \int_{s}^{t } \int_{\partial\Omega}\left( \|\lambda\|_{L^\infty(\partial\Omega)} - \alpha\right) |u|^2 \dd S \dd \tau +
 \liminf_{M \to + \infty} \|u^M(s,\cdot)\|_{L^2(\Omega)}^2 .
 \end{align}
Recall that $u \in C_{\loc}([0,+\infty);w-L^2_{\sigma,\tang}(\Omega))$. This extra regularity enables us to deduce that \eqref{blaDu5} holds true for all $t \geq s$. Finally, since $u^M$ converges to $u$ strongly in $L^2_{\loc}((0,+\infty);L^2_{\sigma,\tang}(\Omega))$ it follows that, again passing to subsequences as needed, $u^M(s)$ converges to $u(s)$ strongly in $L^2_{\sigma,\tang}(\Omega)$ for almost every $s$. This concludes the proof of \eqref{EnInDu}.

Next, we note that, by virtue of Proposition~\ref{EquivWeakFormStokesOp}, we have that the sequence $u^M$ also satisfies
 \begin{align*}
 \int_{\Omega} &\left\{ (\partial_t u^M) \cdot v_j + [(u^M \cdot \nabla) u^M]\cdot v_j \right\}\dd x \\
 & = -\nu \left(\int_\Omega 2Su^M : S v_j \dd x + \int_{\partial\Omega} \alpha u^M \cdot v_j \dd S \right).
 \end{align*}
Once again we multiply this identity by $g^M_j$ and sum in $j$ to obtain
 \begin{equation} \label{blaSu1}
 \frac{d}{dt}\|u^M\|_{L^2(\Omega)}^2 \leq -4\nu \|Su^M\|_{L^2(\Omega)}^2 - 2\nu\int_{\partial\Omega} \alpha |u^M|^2 \dd S.
 \end{equation}
The differential inequality \eqref{blaSu1} is analogous to \eqref{blaDu1} and so we may proceed with the same steps as in the proof of \eqref{EnInDu} to arrive at \eqref{EnInSu}.

Finally, note that $s=0$ is a distinguished time since, by construction, $u^M(0)$ converges strongly to $u(0)=u_0$ when $M\to +\infty$.

This concludes the proof.

\end{proof}

\begin{definition} \label{Leray-Hopf weak solution}
Fix $\alpha \in C^0(\partial\Omega)$ and $u_0 \in L^2_{\sigma,\tang}(\Omega)$. Let $u \in C_{\lo}([0,+\infty);w-L^2_{\sigma,\tang}(\Omega)) \cap L^2_\lo((0,+\infty);H^1_{\sigma,\tang}(\Omega))$, with $\partial_t u \in L^{4/3}_\lo((0,+\infty);(H^1_{\sigma,\tang}(\Omega))^\prime)$, be a weak solution of \eqref{NSwNBC} with initial data $u_0$. We say $u$ is a {\em Leray-Hopf weak solution} if, additionally, for almost every $s \geq 0$, and for $s=0$, and every $t \geq s$, $u$ satisfies both energy inequalities \eqref{EnInDu} and \eqref{EnInSu}.
\end{definition}

\begin{remark}
Since $\partial_t u \in L^{4/3}_{\loc}((0,+\infty);(H^1_{\sigma,\tang}(\Omega))^\prime)$ it is standard to {\em extend} the definition of weak solution so as to allow for time-independent test vector fields $\Phi \in C^\infty_{\sigma,\tang}(\overline{\Omega})$. For such test vector fields the identity \eqref{WeakFormWithSuIdentity} should be substituted by
\begin{align} \label{WeakFormWithSuIdentityPhiTimeIndep}
\int_0^{t} \int_\Omega & [(u \cdot \nabla )\Phi]\cdot u \dd x \dd s + \int_\Omega \Phi(x)\cdot [u_0(x) - u(t,x)] \dd x \nonumber \\ & = \nu\int_0^{t}\left( \int_\Omega 2 S\Phi : Su \dd x + \int_{\partial\Omega} \alpha \Phi \cdot u \dd S\right)\dd s.
\end{align}
We will make use of this alternative formulation later in this work.
\end{remark}

\section{Symmetric Poincar\'e inequality}\label{sec5}

In Theorem~\ref{GlobalExistWeakSol} we obtained the existence of a weak solution to \eqref{NSwNBC} which satisfies two different energy inequalities. In the remainder of this article, we will be showing that, under appropriate hypotheses on the friction coefficient $\alpha$, we can prove exponential decay of the weak solution in energy norm. A crucial tool to prove such results are Poincar\'e-type inequalities. In this section we state and prove such an inequality in terms of the symmetric part of the Jacobian of $u$; this is related to Korn's inequality in $H^1_0$, see \cite[Remark IV.7.3, (IV.87)]{BoyerFabrie2013}. This result is already known, see \cite[Lemma 3.3]{AR2014}; we include a proof for the sake of completeness.

Recall the notation $\Ker S := \{u \in H^1_{\sigma,\tang} (\Omega) \, | \, S u = 0\}$ and $(\Ker S)^\perp := \{ u \in H^1_{\sigma,\tang}(\Omega) \, | \, \int_\Omega u \cdot v \dd x = 0 \text{ for all } v \in \Ker S\}.$

\begin{prop} \label{SymmetricPoincare}
There exists a constant $C=C(\Omega)>0$ such that, for all $u \in (\Ker S)^\perp$, it holds that
\begin{equation} \label{SymPoincIneq}
\|u\|_{L^2(\Omega)} \leq C \| Su\|_{L^2(\Omega)}.
\end{equation}
\end{prop}

\begin{proof}
The proof proceeds in three steps.

\vspace{.3cm}

{\it Step 1.} We claim that there is a constant $K>0$ such that, for all $u \in H^1_{\sigma,\tang}(\Omega)$ it holds that
\[\| Du \|_{L^2(\Omega)} \leq 2 \| Su \|_{L^2(\Omega)} + K\|\lambda\|_{L^\infty(\partial\Omega)}
\| u \|_{L^2(\Omega)}.\]
To see this let $u \in H^1_{\sigma,\tang}(\Omega)$ and observe that, in \eqref{WeakFormStokesOpId}, through a density argument, we can use $\Phi = u$. This eventually yields
\[\int_\Omega |Du|^2 \dd x = 2 \int_\Omega |Su|^2 \dd x + \int_{\partial\Omega} u \cdot d\n(u) \dd S.
\]
Therefore, using the trace inequality, followed by Young's inequality, we obtain
\begin{align*}
\|Du\|_{L^2(\Omega)}^2 & \leq 2 \|Su\|_{L^2(\Omega)}^2 + \|\lambda\|_{L^\infty(\partial\Omega)} \|u\|_{L^2(\partial\Omega)}^2 \\
& \leq 2 \|Su\|_{L^2(\Omega)}^2 + K\|\lambda\|_{L^\infty(\partial\Omega)} \|u\|_{L^2(\Omega)} \|Du\|_{L^2(\Omega)} \\
& \leq 2 \|Su\|_{L^2(\Omega)}^2 + \frac{\left(K\|\lambda\|_{L^\infty(\partial\Omega)} \right)^2}{2}\|u\|_{L^2(\Omega)}^2 + \frac{\|Du\|_{L^2(\Omega)}^2}{2} .
\end{align*}
Hence,
\[\|Du\|_{L^2(\Omega)}^2 \leq 4 \|Su\|_{L^2(\Omega)}^2 + \left(K\|\lambda\|_{L^\infty(\partial\Omega)} \right)^2 \|u\|_{L^2(\Omega)}^2,\]
from which the claim follows easily.

\vspace{.3cm}

{\it Step 2.} Next, we claim that there exists $C>0$ such that, for all $u \in (\Ker S )^\perp$, we have
\[\|u\|_{L^2(\Omega)} \leq C \|Su\|_{L^2(\Omega)} + \frac{1}{2K\|\lambda\|_{L^\infty(\partial\Omega)}} \|Du\|_{L^2(\Omega)},\]
where $K$ is precisely the constant from Step 1.
We argue by contradiction: assume it is not so. Then, for every $N \in \N$, there exists $u_N \in (\Ker S )^\perp$ such that
\[\|u_N\|_{L^2(\Omega)} > N \|Su_N\|_{L^2(\Omega)} + \frac{1}{2K\|\lambda\|_{L^\infty(\partial\Omega)}} \|Du_N\|_{L^2(\Omega)}.\]
Dividing the inequality above by $\|u_N\|_{L^2(\Omega)}$ we may assume, without loss of generality, that $\|u_N\|_{L^2(\Omega)}=1$, so that
\begin{align*}
1 & > N \|Su_N\|_{L^2(\Omega)} + \frac{1}{2K\|\lambda\|_{L^\infty(\partial\Omega)}} \|Du_N\|_{L^2(\Omega)}\\
 & \geq \frac{1}{2K\|\lambda\|_{L^\infty(\partial\Omega)}} \|Du_N\|_{L^2(\Omega)}.
\end{align*}
In particular, this means that $\{u_N\}$ is a bounded sequence in $H^1_{\sigma,\tang}(\Omega)$. Thus, passing to a subsequence as needed, we may assume that $u_N$ converges weakly in $H^1_{\sigma,\tang}(\Omega)$ to some $u$. Since $H^1_{\sigma,\tang}(\Omega)$ is compactly imbedded in $L^2_{\sigma,\tang}(\Omega)$ it follows that $\|u\|_{L^2(\Omega)}=1$. However, since
\begin{align*}
\frac{1}{N} & > \|Su_N\|_{L^2(\Omega)} + \frac{1}{2K\|\lambda\|_{L^\infty(\partial\Omega)}N} \|Du_N\|_{L^2(\Omega)}\\
& \geq \|Su_N\|_{L^2(\Omega)},
\end{align*}
it follows that $Su = 0$. But originally we had $u_N\in (\Ker S )^\perp$ so, since $(\Ker S)^\perp$ is a {\em closed} subspace of $H^1_{\sigma,\tang}(\Omega)$ (and of $L^2_{\sigma,\tang}(\Omega)$), it follows that $ u \in \Ker S \cap \Ker S^\perp$, {\it i.e.} $u=0$, a contradiction with $\|u\|_{L^2}=1$.

\vspace{.3cm}

{\it Step 3.} We put together the results in Steps 1 and 2 to conclude:
\begin{align*}
\|u\|_{L^2(\Omega)} & \leq C \|Su\|_{L^2(\Omega)} + \frac{1}{2K\|\lambda\|_{L^\infty(\partial\Omega)}} \|Du\|_{L^2(\Omega)} \\
& \leq C \|Su\|_{L^2(\Omega)} + \frac{1}{2K\|\lambda\|_{L^\infty(\partial\Omega)}}\left( 2 \| Su \|_{L^2(\Omega)} + K\|\lambda\|_{L^\infty(\partial\Omega)}\| u \|_{L^2(\Omega)}\right)\\
& \leq \left(C + \frac{1}{K\|\lambda\|_{L^\infty(\partial\Omega)}} \right)\| Su \|_{L^2(\Omega)} +
\frac{\| u \|_{L^2(\Omega)}}{2}.
\end{align*}

This yields the desired estimate and concludes the proof.

\end{proof}

\section{Exponential decay -- Part 1}\label{sec6}

In this section, we put together the energy inequalities in Theorem~\ref{GlobalExistWeakSol} and the symmetric Poincar\'e-type inequality in Proposition~\ref{SymmetricPoincare} to obtain our first exponential decay results, using both energy inequalities and following the path in which such estimates are obtained in the no-slip case.

\begin{theorem}
 Let $\Omega$ be a bounded, connected open set in $\R^3$, with smooth boundary. Consider $u_0 \in L^2_{\sigma,\tang}(\Omega)$ and let
 $u$ be a Leray-Hopf weak solution of the incompressible Navier-Stokes equations \eqref{NSwNBC} with Navier boundary conditions and initial data $u_0$, according to Definition~\ref{Leray-Hopf weak solution}. Then we have:
 \begin{enumerate}
 \item If $\Omega$ is such that $\Ker S = \{0\}$ and if the friction coefficient $\alpha = \alpha(x)\geq 0$ for all $x \in \partial\Omega$, then $u \to 0$ strongly in $L^2(\Omega)$, exponentially fast, as $t \to +\infty$. More precisely, there exists $C>0$ such that $\|u(t)\|_{L^2(\Omega)}\leq \|u_0\|_{L^2(\Omega)}\exp (-C\nu t)$.
 \item If the friction coefficient $\alpha=\alpha(x)>0$, $x \in \partial\Omega$, then, with no further restrictions on $\Omega$, the same conclusion above holds true.
 \end{enumerate}
 \end{theorem}

\begin{proof}
 Let us begin by assuming that $\Omega$ is such that $\Ker S = \{0\}$. In this case we can use the Poincar\'e-type inequality in Proposition~\ref{SymmetricPoincare}. Recall the energy inequality \eqref{EnInSu}, valid for a.e. $s \geq 0$ and for $s=0$ and for every $t \geq s$, which we rewrite as
 \[\|u(t) \|_{L^2}^2 \leq \|u(s)\|_{L^2}^2 - 4\nu \int_s^t \|Su\|_{L^2}^2 \dd \tau - 2\nu \int_s^t \int_{\partial\Omega} \alpha |u|^2 \dd S \dd \tau.\]
 Now, under the additional assumption that $\alpha \geq 0$ and using \eqref{SymPoincIneq}, it follows that
 \[\|u(t) \|_{L^2}^2 \leq \|u(s)\|_{L^2}^2 - \frac{4\nu}{C} \int_s^t \|u(\tau)\|_{L^2}^2 \dd \tau .\]
 Use the version of Gr\"onwall's inequality in Proposition~\ref{Edriss} with $y(t) = \|u(t) \|_{L^2}^2$ and $K=4 \nu / C$ to conclude the proof of item {\it (1)}.

 Next, assume only that $\alpha > 0$. Let $\eta \in (0,1)$. Taking a convex combination of the energy inequalities \eqref{EnInDu} and \eqref{EnInSu} produces the estimate
 \begin{align} \label{ConvCombEnIn}
 \|u(t) \|_{L^2}^2 & + 2\nu \eta \int_s^t \|Du\|_{L^2}^2 \dd \tau + 4\nu (1-\eta) \int_s^t \|Su\|_{L^2}^2 \dd \tau \nonumber \\
 & \leq \|u(s)\|_{L^2}^2 + 2\nu \eta \int_s^t \int_{\partial\Omega} \left(\|\lambda\|_{L^\infty(\partial\Omega)} - \alpha\right) |u|^2 \dd S \dd \tau
 \nonumber \\
 & - 2\nu (1-\eta) \int_s^t \int_{\partial\Omega} \alpha |u|^2 \dd S \dd \tau \nonumber \\
 & = \|u(s)\|_{L^2}^2 - 2\nu \int_s^t \int_{\partial\Omega} \left(\alpha - \eta \|\lambda\|_{L^\infty(\partial\Omega)}\right) |u|^2 \dd S \dd \tau.
 \end{align}
 Again, this is valid for a.e. $s \geq 0$ and for $s=0$ and for every $t \geq s$.

 Since $\alpha \in C^0(\partial \Omega)$, with $\partial \Omega$ compact, and because $\alpha > 0$, we can choose
 \[0< \eta = \min \left\{ \frac{\min_{x \in \partial \Omega} \alpha (x)}{\|\lambda\|_{L^\infty(\partial\Omega)} }, \frac{1}{2} \right\} < 1.\]
 This allows us to discard the term with the boundary integral.

 We may thus re-write \eqref{ConvCombEnIn} as
 \begin{equation*}
 \|u(t) \|_{L^2}^2 \leq \|u(s)\|_{L^2}^2 - 2\nu \eta \int_s^t \|Du\|_{L^2}^2 \dd \tau - 4\nu (1-\eta) \int_s^t \|Su\|_{L^2}^2 \dd \tau,
 \end{equation*}
 so that, discarding additionally the term with the symmetric derivative and using the classical Poincar\'e inequality, valid since $u$ has vanishing mean, we deduce that
 \begin{equation*}
 \|u(t) \|_{L^2}^2 \leq \|u(s)\|_{L^2}^2 - \frac{2\nu \eta}{C} \int_s^t \|u(\tau)\|_{L^2}^2 \dd \tau.
 \end{equation*}

Using once more Proposition~\ref{Edriss} with $y(t) = \|u(t) \|_{L^2}^2$ and $K=2 \nu \eta / C$ we establish item {\it (2)}.

 This concludes the proof.

 \end{proof}

\section{Exponential decay -- Part 2}\label{sec7}

In this last section we will address the large time behavior of solutions of the incompressible Navier-Stokes equations with Navier boundary conditions when the friction coefficient $\alpha$ vanishes identically on the boundary of the domain. We have already considered this for domains $\Omega$ such that $\Ker S = \{0\}$, so we now concentrate only on fluid domains for which $\Ker S \neq \{0\}$. Recall that $\Ker S$ was defined as a subspace of $H^1_{\sigma,\tang}(\Omega)$ so, in particular, vector fields in $\Ker S$ must be tangent to $\partial \Omega$. As usual we assume $\Omega$ is a bounded, smooth, open set in $\R^3$, with smooth boundary and not necessarily simply connected.

We begin with a well-known elementary characterization of vector fields $w=w(x)$ in $\R^3$ for which $S w = 0$.

\begin{lemma} \label{KerS=0}
 Let $w=w(x) \in H^1(\Omega)$ be a vector field such that $Sw = 0$ for all $x\in \Omega$. Then there exist constant vectors $a$, $b \in \R^3$ such that
 \[w(x) = a + b \wedge x.\]
 Moreover, $\ds{\int_{\Omega} w \dd x = a}$ and $\curl w = b$.
\end{lemma}

The proof of Lemma~\ref{KerS=0} can be found in \cite[Lemma IV.7.5]{BoyerFabrie2013}; see also \cite[Chapter 1, Lemma 1.1]{Temam1983Plasticity}.

Equivalently stated, the result in Lemma~\ref{KerS=0} is that an $H^1(\Omega)$ vector field $w$ for which $Sw=0$ is the infinitesimal generator of the motion of a rigid body; that is, translation and rotation about an (at least one) axis. With this in mind it is intuitively clear that if, additionally, such a vector field is non-zero and tangent to the boundary of a {\em bounded} domain $\Omega$, then $\Omega$ is a solid of revolution around an axis. In the following proposition we will formalize this statement along with a partial converse.

\begin{prop} \label{eejit}
Assume that $\Omega$ is a bounded, smooth, connected domain in $\R^3$.
\begin{enumerate}
 \item Let $b \in \R^3$, $b \neq 0$. If $\Omega$ is a solid of revolution with symmetry axis $s \mapsto a + sb$, $s\in \R$, then
 $(b\wedge (x-a)) \cdot \n = 0$ for every $x \in \partial \Omega$, $\n = \n(x)$.
 \item Conversely, let $a$, $b \in \R^3$, $b \neq 0$, and consider the vector field $w=w(x)=a+b\wedge x$. Assume that $w \cdot \n = 0$ for every $x \in \partial \Omega$, $\n = \n(x)$. Then $w(x) = b \wedge (x-c)$ for some $c\in\R^3$ and $\Omega$ is a solid of revolution with symmetry axis $s \mapsto c+ sb$, $s \in \R$.
\end{enumerate}
\end{prop}

\begin{proof}
 To see the first statement let us take, without loss of generality, $a=(0,0,0)$ and $b=(0,0,1)= \e_3$, otherwise change variables. Assume that $\Omega$ is a solid of revolution with respect to the $z$-axis.
 Equivalently, we suppose that (the components of) $\partial \Omega$ is (are concentric surfaces) given by $f(\sqrt{x_1^2 + x_2^2}, x_3) = 0$, for some smooth real-valued function(s) $f$ for which $0$ is a regular value. The normal vector, at any point on $\partial\Omega$, is, hence, a linear combination of $x_H:=(x_1,x_2,0)$ and $\e_3$. The desired conclusion follows once we observe that $b \wedge x \equiv \e_3 \wedge x = (-x_2,x_1,0)$.

 We introduce the notation $x_H^\perp :=(-x_2,x_1,0)$.

 For the second statement let us, again, assume without loss of generality that $b=\e_3$, otherwise we choose a different coordinate system.

 Assume, first, that $a \cdot b = 0$. In this case there exists $c\in \R^3$ such that $a=b\wedge c$ and we may translate our coordinate system so as to assume, again without loss of generality, that $a=0$. Summarizing, we wish to show that, if $[\e_3 \wedge x ]\cdot \n = 0$ on $\partial\Omega$, then $\Omega$ is invariant under rotation around the $z$-axis. Since $\Omega$ is smooth it follows that (each component of) $\partial \Omega$ is a level set of a smooth real-valued function $f$ at a regular value of $f$; this is a consequence of the Collar Neighborhood theorem. In particular, $\nabla f \neq 0$ is smooth and parallel to $\n$. Let $\hat{x_H}:=x_H/|x_H|$ and $\hat{x_H^\perp}:=x_H^\perp/|x_H|$. Then $\nabla f(x)$ may be decomposed uniquely as a linear combination of $\hat{x_H}$, $\hat{x_H^\perp}$ and $\e_3$ and, if $0 = [\e_3 \wedge x] \cdot \n = x_H^\perp \cdot \n$, it follows that $\nabla f \cdot \hat{x_H^\perp} = 0$. This means there is no azimuthal component of $\nabla f$, that is, $f(x)=f(\sqrt{x_1^2 + x_2^2}, x_3)$. Thus $\Omega$ is invariant under rotation around the $z$-axis, as desired.

 Lastly, suppose $a \cdot b \neq 0$. Writing $a=a_H + a_3 b$ and translating away $c$ such that $a_H=b\wedge c$ we can assume further, without loss of generality, that $a=a_3 b= (0,0,a_3)$, with $a_3 \neq 0$. From $[a+b\wedge x]\cdot \n = 0$ it follows that $(a_3 \e_3 + x_H^\perp ) \cdot \nabla f = 0$. Hence $a_3 \partial_{x_3} f + \nabla f \cdot x_H^\perp = 0$ on $\partial \Omega$. Since $\Omega$ was assumed to be {\em bounded} it follows that there are, at least, two points $P_1$ and $P_2$ on $\partial\Omega$ at which $\nabla f$ is parallel to $b=\e_3$. In particular, $\nabla f \cdot x_H^\perp = 0$ at $P_1$ and $P_2$. It follows that $a_3 \partial_{x_3}f = 0$ at $P_1$ (and at $P_2$) and, since $a_3 \neq 0$, $\partial_{x_3}f$ vanishes at $P_1$ (and at $P_2$). Since $\nabla f$ is parallel to $\e_3$ at $P_1$ (and at $P_2$) we conclude that $\nabla f(P_1)=0$ (and $\nabla f(P_2)=0$ as well), which is not possible. We deduce that this last case does not arise and, with this, we conclude the proof.

\end{proof}

The following result is an immediate consequence of Lemmas~\ref{KerS=0} and \ref{eejit}. This result may also be found in \cite[Theorem 1]{FG2022Annali}.

\begin{corollary} \label{immediate}
 Let $\Omega$ be a bounded, smooth, connected domain in $\R^3$. Then
 \begin{enumerate}
 \item $\Ker S = \{0\}$ if $\Omega$ is not invariant under rotation around an axis;
 \item $\mathrm{dim } \,\Ker S = 1$ if $\Omega$ is invariant under rotation around a single axis;
 \item $\mathrm{dim } \,\Ker S = 3$ if $\partial \Omega$ is a (are concentric) sphere(s).
 \end{enumerate}
\end{corollary}

We introduced $\Ker S$ as a subspace of $H^1_{\sigma,\tang}(\Omega)$. We wish to consider the natural extension of $S$ to $L^2_{\sigma,\tang}(\Omega)$, with values in $H^{-1}(\Omega)$; it's kernel, a subspace of $L^2_{\sigma,\tang}(\Omega)$, will still be denoted
$\Ker S$. We consider the orthogonal decomposition $L^2_{\sigma,\tang}(\Omega) = (\Ker S)^\perp \oplus \Ker S$, with respect to the $L^2$-inner product. For each $v \in L^2_{\sigma,\tang}(\Omega)$ we denote the $L^2$-projection of $v$ onto $\Ker S$ by $\mathrm{Proj}_{\Ker S} v$.

The proposition below actually encompasses two facts in the case $\alpha = 0$. The first one is that infinitesimal generators of rigid rotations are stationary solutions of \eqref{NSwNBC}. This is not surprising, given the physics of the problem. The second fact, which is not obvious, is that for any weak solution $v$ of \eqref{NSwNBC}, satisfying \eqref{WeakFormWithSuIdentityPhiTimeIndep}, we have that $\mathrm{Proj}_{\Ker S} v$ is a conserved quantity.

\begin{prop} \label{stationary}
Let $u$ be a weak solution of the incompressible Navier-Stokes equations with Navier boundary conditions, \eqref{NSwNBC} with vanishing friction coefficient. Then the vector field $\mathrm{Proj}_{\Ker S} u$ is a stationary weak solution of \eqref{NSwNBC}.
\end{prop}

\begin{proof}
 Let us assume that $\Ker S \neq \{0\}$, otherwise the result is trivial.

 By Proposition~\ref{eejit} this means that $\Omega$ is a rotationally invariant domain around some axis of symmetry. From Corollary~\ref{immediate} we have $\Ker S$ is either $1$-dimensional or $3$-dimensional.

 Let us assume, first, that $\mathrm{dim }\,\Ker S = 1$, so that $\Omega$ is invariant around a single axis. Let $b$ be a unit vector in the direction of the axis of symmetry of $\Omega$. We may assume, as usual, that $b = \e_3$. Recall Lemma~\ref{KerS=0}, from which we deduce, together with the proof of Proposition~\ref{eejit}, that
 $\Ker S = \{ \beta b \wedge (x - c), \, \beta \in \R\}.$ We may assume, without loss of generality, that $c=0$, by translating the coordinate system. In this case $\Ker S = \{ \beta x_H^\perp, \, \beta \in \R\}.$ Let $C=C_\Omega := (\|x_H^\perp\|_{L^2(\Omega)}^2)^{-1}$. Then
 \[\mathrm{Proj}_{\Ker S} u = C\left(\int_\Omega u\cdot x_H^\perp \dd x \right)x_H^\perp.\]
 Let us denote $W:= \mathrm{Proj}_{\Ker S} u$. We will show that $W$ satisfies the weak formulation provided in \eqref{WeakFormWithSuIdentityPhiTimeIndep}, with $\alpha = 0$, and that $W(t,\cdot) \equiv W_0(\cdot)$. In other words, for any $\Phi \in C^\infty_{\sigma,\tang}(\overline{\Omega})$, we show that
 \begin{align} \label{Weq}
 \int_0^{t} \int_\Omega & [(W \cdot \nabla )\Phi]\cdot W \dd x \dd s + \int_\Omega \Phi(x)\cdot [W_0(x) - W(t,x)] \dd x \nonumber \\
 & = 2\nu\int_0^{t}\left( \int_\Omega S\Phi : SW \right) \dd x \dd s.
\end{align}
We identify each of the three terms above.

First we observe that, since $W \in \Ker S$, it is immediate that the right-hand-side term of \eqref{Weq} vanishes.

Next, because $x_H^\perp$ is smooth and tangent to $\partial\Omega$, it is possible to integrate by parts, in $x$, the nonlinear term, obtaining:
\begin{equation*}
\int_0^{t} \int_\Omega [(W \cdot \nabla )\Phi]\cdot W \dd x \dd s = - \int_0^{t} \int_\Omega [(W \cdot \nabla )W]\cdot \Phi \dd x \dd s.
\end{equation*}
A direct calculation yields
\[x_H^\perp \cdot \nabla x_H^\perp = -x_H = - \nabla \left(\frac{|x_H|^2}{2}\right).\]
Therefore, since $\dv \Phi = 0$ and $\Phi$ is tangent to $\partial\Omega$, it follows that the nonlinear term in \eqref{Weq} also vanishes.

Lastly, we will show that the second term on the left-hand-side of \eqref{Weq} vanishes, thereby establishing \eqref{Weq}.

Recall $u$ is a weak solution of \eqref{NSwNBC}, thus it satisfies \eqref{WeakFormWithSuIdentityPhiTimeIndep}, with $\alpha = 0$, for any test vector field in $C^\infty_{\sigma,\tang}(\overline{\Omega})$. We use $\Phi = x_H^\perp \in C^\infty_{\sigma,\tang}(\overline{\Omega})$ and we note that, just as for $W$, $ S x_H^\perp= 0$. Additionally, it is straightforward to verify that the nonlinear term $[(u \cdot \nabla)x_H^\perp ]\cdot u = 0$. Using this information in \eqref{WeakFormWithSuIdentityPhiTimeIndep} leaves us with
\begin{equation} \label{tindepend}
\int_\Omega x_H^\perp \cdot [u_0(x) - u(t,x)] \dd x = 0.
\end{equation}
Clearly, this implies that $W(t,x)=W_0(x)$, $x \in \Omega$. Thus $W$ is a stationary (weak) solution of \eqref{NSwNBC}.
This concludes the proof in the case $\mathrm{dim } \, \Ker S = 1$.

The remaining case, $\mathrm{dim } \, \Ker S = 3$, corresponds to $\partial\Omega$ being a sphere or concentric spheres. Without loss of generality we assume, again, that the center of the sphere or concentric spheres is $c=0$. In this case all three unit vectors $\e_1$, $\e_2$ and $\e_3$ are directions of axes of symmetry of $\Omega$ and, using Lemma~\ref{KerS=0} and Proposition~\ref{eejit}, we obtain that $\Ker S$ is generated by $\{\e_1 \wedge x, \e_2 \wedge x, \e_3 \wedge x\}$.
Writing explicitly each of these vector products we have
\[\Ker S = \{\alpha (0,-x_3,x_2) + \beta (x_3,0,-x_1) + \gamma (-x_2,x_1,0), \, \alpha, \,\beta, \, \gamma \in \R \}.
\]

Let us introduce the notations $Y_1:=(0,-x_3,x_2)$, $Y_2 :=(x_3,0,-x_1)$ and $Y_3 :=(-x_2,x_1,0)$; note that $x_H^\perp = Y_3$. Furthermore, by symmetry, $Y_i$ and $Y_j$ are $L^2$-orthogonal if $i \neq j$. We find, hence,
\begin{align*}
\mathrm{Proj}_{\Ker S} u &= \sum_{i=1}^3 \frac{1}{\|Y_i\|_{L^2}}\left(\int_\Omega u\cdot Y_i \dd x \right)Y_i.
\end{align*}
We want to show that $\mathrm{Proj}_{\Ker S} u$ is a time-independent weak solution of \eqref{NSwNBC} with $\alpha = 0$. As before, let $W:=\mathrm{Proj}_{\Ker S} u$ and consider identity \eqref{Weq}, with $\Phi \in C^\infty_{\sigma,\tang}(\overline{\Omega})$. Clearly the right-hand-side, once again, vanishes. To show that the nonlinear term vanishes it is enough to show that, for any $\alpha=\alpha(t)$, $\beta=\beta(t)$, $\gamma=\gamma(t)$,
\[[(\alpha Y_1 + \beta Y_2 + \gamma Y_3 )\cdot \nabla ] ( \alpha Y_1 + \beta Y_2 + \gamma Y_3 )\]
is a gradient vector field, something which can be easily explicitly checked; we omit the calculation. Lastly, we consider the second term on the left-hand-side of \eqref{Weq}. To conclude the proof that $W$ is a stationary weak solution it is enough to show that
\begin{equation} \label{Yi}
\left(\int_\Omega u\cdot Y_i \dd x \right)Y_i \quad \text{ is time-independent, for } i=1,2,3.
\end{equation}
The proof of \eqref{Yi} is the same as the proof of \eqref{tindepend} in the case
$\mathrm{dim }\,\Ker S = 1$, using, instead of $\Phi = x_H^\perp$, the test vector fields $\Phi = Y_i \in C^\infty_{\sigma,\tang}(\overline{\Omega})$, $i=1,2,3$.

\end{proof}

Finally, still in the frictionless case $\alpha = 0$, we prove decay of the weak solution with initial data $u_0$ to the steady rigid rotation given by $\mathrm{Proj}_{\Ker S} u_0$.

\begin{theorem}
 Let $u$ be a Leray-Hopf weak solution of \eqref{NSwNBC} with friction coefficient $\alpha = 0$. Then $u \to \mathrm{Proj}_{\Ker S} u_0$ exponentially fast as $t \to +\infty$. More precisely, there exist $C>0$ such that
 \begin{equation} \label{expdecaytoindata}
 \|u(t)- \mathrm{Proj}_{\Ker S} u_0\|_{L^2(\Omega)}\leq \|u_0-\mathrm{Proj}_{\Ker S} u_0\|_{L^2(\Omega)}\exp (-C\nu t).
 \end{equation}
\end{theorem}

\begin{proof}
We begin by recalling the energy inequality \eqref{EnInSu}, substituting $\alpha = 0$:
\begin{equation} \label{EnInSuAlpha0}
\|u(t) \|_{L^2}^2 + 4\nu \int_s^t \|Su\|_{L^2}^2 \dd \tau \leq \|u(s)\|_{L^2}^2,
\end{equation}
for a.e. $s \geq 0$ and for $s=0$ and for every $t \geq s$
Next we note that, since $u - \mathrm{Proj}_{\Ker S} \, u$ is $L^2$-orthogonal to $\mathrm{Proj}_{\Ker S} u$ and because
$ S \left(\mathrm{Proj}_{\Ker S} u \right) = 0$, \eqref{EnInSuAlpha0} can be re-written as
\begin{align*}
\|u(t) - \mathrm{Proj}_{\Ker S} u (t)\|_{L^2}^2 & + \|\mathrm{Proj}_{\Ker S} u(t)\|_{L^2}^2 \\
&+ 4\nu \int_s^t \|S\left( u - \mathrm{Proj}_{\Ker S} u\right)\|_{L^2}^2 \dd \tau \\
& \leq \|u(s)- \mathrm{Proj}_{\Ker S} u(s)\|_{L^2}^2 + \|\mathrm{Proj}_{\Ker S} u(s)\|_{L^2}^2.
\end{align*}
In view of Proposition~\ref{stationary} this inequality amounts to
\begin{align*}
\|u(t) - \mathrm{Proj}_{\Ker S} u (t)\|_{L^2}^2 &+ 4\nu \int_s^t \|S\left( u - \mathrm{Proj}_{\Ker S} u\right)\|_{L^2}^2 \dd \tau \\
& \leq \|u(s)- \mathrm{Proj}_{\Ker S} u(s)\|_{L^2}^2.
\end{align*}
Use the symmetric Poincar\'e-type inequality \eqref{SymmetricPoincare} to find
\begin{align*}
\|u(t) - &\mathrm{Proj}_{\Ker S} u (t)\|_{L^2}^2 \\
&\leq \|u(s)- \mathrm{Proj}_{\Ker S} u(s)\|_{L^2}^2 - C\nu \int_s^t \| u(\tau) - \mathrm{Proj}_{\Ker S} u(\tau)\|_{L^2}^2 \dd \tau.
\end{align*}
Finally, using Proposition~\ref{Edriss} with $y(t) = \|u(t) - \mathrm{Proj}_{\Ker S} u (t)\|_{L^2}^2 $ and $K=C\nu$ allows us to deduce \eqref{expdecaytoindata} and conclude the proof.

\end{proof}

The result above was originally obtained in \cite[Theorem 6.2]{Watanabe2003}, albeit with a different proof. In particular, the aforementioned proof did not involve the result established in Proposition~\ref{stationary}, namely, conservation of $\mathrm{Proj}_{\Ker S} u (t)$.

\section{Comments and conclusions}\label{sec8}

In this section we summarize what has been accomplished in this article, discuss the connection with related work, formulate a few open problems and discuss directions for future investigation.

Our main results are the existence of a Leray-type weak solution, with two versions of the corresponding energy inequality, and three long-time exponential decay estimates. The basic structure of the arguments are classical. Still, aside from providing a comprehensive account of existence and large-time behavior of weak solutions for Navier-Stokes with Navier boundary conditions, the main point of this work is to account for the influence of the differential geometry of the boundary on this problem. This arises in several moments:

\begin{itemize}
\item We use Lemma~\ref{bdrytermNBC} to obtain a weak formulation of the Navier boundary condition using the shape operator of the boundary, see Definition~\ref{WeakFormWithDu}.

\item We again use Lemma~\ref{bdrytermNBC} to define the Stokes operator $\mathbb{A}$ in \eqref{StokesOp} and we use estimates on the principal curvatures of the boundary to prove its boundedness and coercivity, and the self-adjointness of the shape operator to prove that $\mathbb{A}$ itself is self-adjoint. This is needed for the construction of the basis of eigenfunctions in Proposition~\ref{eigenfunctionsStokes}.

\item Aside from the presence of the shape operator in Definition~\ref{WeakFormWithDu}, used throughout in Theorem~\ref{ExistWeakSolWithDu}, we again use the shape operator to express the dissipation term in \eqref{NSwNBC} in terms of the symmetric gradient in Proposition~\ref{EquivWeakFormStokesOp}. This allows us to rewrite identity \eqref{WeakFormWithDuIdentity} as \eqref{WeakFormWithSuIdentity}, which does not depend explicitly on the geometry of the boundary.

\item The weak solutions obtained satisfy two energy inequalities, namely \eqref{EnInDu} and \eqref{EnInSu}, where only the former depends explicitly on the geometry of the boundary, through the bound on the principal curvatures.

\item The constant in the symmetric Poincar\'{e} inequality in Proposition~\ref{SymmetricPoincare} depends on the bounds on the principal curvatures.

\item The proof of the exponential decay in the case $\alpha > 0$ uses {\it both} energy identities, juggling one against the other.

\item The decay to steady state for domains which are solids of revolution does not involve the differential geometry of the boundary explicitly, as it relies on the energy identity \eqref{EnInSu} for the decay, but it still uses the symmetric Poincar\'{e} inequality Proposition~\ref{SymmetricPoincare}. Furthermore, we make essential use of the geometry of the domain to characterize the vector fields in $\Ker(S)$.
\end{itemize}

From the discussion above we conclude that the geometric identity expressed in Lemma~\ref{bdrytermNBC} is a key part of the present work.

Much of our analysis is inspired by the two-dimensional work done by Clopeau {\it et al} in \cite{CMR}. Notably, the formulation of the Navier boundary conditions which makes explicit the influence of the geometry of the boundary is already present: see \cite[Lemma 2.1]{CMR}. Furthermore, their work depends on the construction of an appropriate Galerkin basis, which inspired the corresponding construction presented here. We note that, in \cite{FG2022DCDS}, an outline of the Galerkin approximation was obtained in the case $\alpha = 0$ and for three-dimensional Lipschitz domains.

We conclude this section with a discussion of future lines of research and open problems.

In a forthcoming paper, the authors study the two-dimensional problem, exploring existence of {\it strong} solutions and exponential decay in a higher norm. Our objective is to extend previous analysis by Kelliher in \cite{KNavier}, complementing the analysis done for domains with holes and adding the discussion on exponential decay.

One interesting special case which we have left open in our analysis is exponential decay for domains which are solids of revolution and with friction coefficient $\alpha \geq 0$. Technically, our work does not extend to this situation, so a new idea is needed.

Lastly, given that dissipation is, in general, due to a combination of boundary friction and viscosity, it is natural to ask whether decay might still be true in situations where $\alpha$ is allowed to be negative. This is a case which might arise in flows with an active boundary, see \cite{GS22c}.

\appendix
\section{}

In this Appendix we state and prove a version of Gronwall's inequality which is key to our exponential decay results. A special case of this result is
implicitly contained in the proof of \cite[Theorem 3.5.1]{SohrBook}.

\begin{prop}\label{Edriss}
Let $y \in L^1_\loc [0,+\infty)$ be a nonnegative function and let $K > 0$. Assume that, for almost every $s\geq 0$ and for every $t \geq s$, it holds that
\begin{equation}\label{IntegIneq}
 y(t) \leq y(s) - K\int_s^t y(\xi) \dd \xi.
\end{equation}
Assume, additionally, that \eqref{IntegIneq} holds for $s=0$.

Then
\[y(t) \leq y_0 \, e^{-Kt} \text{ for all } t \geq 0 .\]
\end{prop}

\begin{proof}
Let $E \subset (0,+\infty)$ be such that $|E|=0$ and \eqref{IntegIneq} holds for every
$s \in E^c$. In particular we have, for every $s \in E^c$ and every $t \geq s$,
\begin{equation}\label{nonincreasey}
y(t) \leq y(s).
\end{equation}

Fix $\delta > 0$.

Let $X_0 = [0,\delta] \cap E^c$. Clearly $|X_0| = \delta$.

We define, recursively, the sets
\[X_n = \{ \tau \in X_{n-1} \text{ such that } \tau + n \delta \in E^c \}.\]
Observe that $X_n = X_{n-1} \cap \{\rho - n\delta \text{ such that } \rho \in E^c \cap [n\delta, (n+1)\delta] \}$.

We have, inductively, that $|X_n|=\delta$. Indeed, we already know $|X_0|=\delta$. Suppose now that $|X_{n-1}|=\delta$. Clearly, $|E^c \cap [n\delta, (n+1)\delta]|=\delta$. It follows immediately that $|X_n|=\delta$ as the intersection of two subsets of $[0,\delta]$ of total measure.

Next observe that the sets $\{X_n\}$ are nested:
\[\ldots X_n \subset X_{n-1} \subset X_{n-2} \subset \ldots \subset X_0.\]

Set
\[X_\infty \equiv \cap_{n=0}^\infty X_n.\]
Then $|X_\infty| = \lim_{n\to \infty} |X_n| = \delta$. Consequently $X_\infty \neq \emptyset$.

It is easy to see that
\[X_\infty = \{\tau \in [0,\delta] \text{ such that } \tau + k \delta \in E^c \text{ for all } k = 0, 1, 2, \ldots \}.\]

Fix $\tau \in X_\infty$. Let $n \in \{0, 1, 2, \ldots\}$. Since $\tau + n\delta \in E^c$, it follows from our hypothesis \eqref{IntegIneq} that
\begin{equation}\label{IntegIneq1}
 y(\tau + (n+1)\delta) \leq y(\tau + n\delta) - K\int_{\tau + n\delta}^{\tau + (n+1)\delta} y(\xi) \dd \xi.
\end{equation}
From \eqref{nonincreasey} we obtain
\begin{equation} \label{nonincreasey2}
y(\tau + (n+1)\delta) \leq y(\xi) \text{ for every } \xi \in E^c \cap [\tau + n\delta,\tau + (n+1)\delta].
\end{equation}
Since the set $\{\xi \in [\tau + n\delta, \tau + (n+1)\delta] \text{ such that \eqref{nonincreasey2} does not hold} \}$ is contained in $E \cap [\tau + n\delta,\tau + (n+1)\delta]$, which has measure zero, it follows that
\begin{equation} \label{nonincreasey3}
 -K\int_{\tau + n\delta}^{\tau + (n+1)\delta} y(\xi) \dd \xi \leq -K \, \delta \, y(\tau + (n+1)\delta).
 \end{equation}
Inserting \eqref{nonincreasey3} into \eqref{IntegIneq1} and moving terms around we find
\begin{equation}\label{IntegIneq2}
y(\tau + (n+1)\delta) \leq \frac{1}{1+K\delta} \, y(\tau + n\delta).
\end{equation}

Set
\[\theta = \frac{1}{1+K\delta}\]
and iterate \eqref{IntegIneq2} backwards to deduce that
\begin{equation} \label{IntegIneq3}
y(\tau + (n+1)\delta) \leq \theta^{n+1} y(\tau), \quad n = 0, 1, 2, \ldots.
\end{equation}
Of course \eqref{IntegIneq3} holds trivially for $n=-1$ so that
\begin{equation} \label{IntegIneq4}
y(\tau + m\delta) \leq \theta^{m} y(0), \quad m = 0, 1, 2, \ldots.
\end{equation}
where we used, additionally, \eqref{nonincreasey} with $s=0$ since $0 \in E^c$.

Let $t \geq \tau$. Then there exists $m \in \{0, 1, 2, \ldots\}$ such that
\[\tau + m\delta \leq t \leq \tau + (m+1)\delta.\]
Because $0 < \theta < 1$ we have $\theta^m \leq \theta^{\frac{t-\tau}{\delta} - 1}$. In addition, since $\tau + m\delta \in E^c$, $y(t) \leq y(\tau + m\delta)$ by \eqref{nonincreasey}. We use these estimates in \eqref{IntegIneq4} to get
\[y(t) \leq \theta^{\frac{t-\tau}{\delta} - 1} y(0), \text{ for all } t \geq \tau.\]
We re-write this as
\begin{equation}\label{IntegIneq5}
y(t) \leq \left(\frac{1}{1+K\delta}\right)^{(t-\tau)/\delta}(1+K\delta) \, y(0), \text{ for all }
t \geq \tau.
\end{equation}
Recall that $X_\infty \subset [0,\delta]$ is a set of full measure. Therefore there exists a sequence $\{\tau_j\} \subset X_\infty$ such that $\tau_j \to 0$. Furthermore, the estimate \eqref{IntegIneq5} is true for $\tau = \tau_j$, $t \geq \tau_j$, for all $j$. Therefore, passing to the limit $j \to \infty$ gives
\begin{equation}\label{IntegIneq6}
y(t) \leq (1+K\delta)^{-t/\delta} (1+K\delta) \, y(0), \text{ for all }
t \geq 0.
\end{equation}
Letting $\delta \to 0$ in \eqref{IntegIneq6} we conclude that
\[y(t) \leq y(0) e^{-K t} \text{ for all } t \geq 0,\]
as desired.

This concludes the proof.

\end{proof}

\begin{proof}[\textbf{Second proof of \cref{Edriss}}]
	Define the function $x \colon [0, \iny) \to [0, \iny)$ by
	\begin{align*}
		x(s)
			&= \sup \set{y(t) \colon t > s}.
	\end{align*}
	Then $x$ is decreasing (meaning non-increasing) by its definition,
	and we we will show that
	$x$ is right-continuous, $x = y$ a.e., and $x'(t) \le - K x(t)$ a.e..

	We first show that $x$ is right-continuous.
	
	Let $J$ be a set of full measure in $[0, \iny)$ for which
	\cref{IntegIneq} holds for every $s \in J$ and for every $t \geq s$.
	From \cref{IntegIneq},
	\begin{align} \label{e:dec}
		s \in J, \, t \ge s \implies y(t) \le y(s),
	\end{align}
	which gives that for $s \in J$,
	\begin{align}\label{e:Msup}
		x(s)
			= \sup \set{y(t) \colon t \in J, t > s}.
	\end{align}
	If $s \notin J$ and $t > s$ then $s < t' < t$ for some $t' \in J$,
	and $y(t') \ge y(t)$. This shows that \cref{e:Msup} holds also
	for $s \in J^c$, giving \cref{e:Msup} for all $s$.
	
	Then, for any $t \in [0, \iny)$,
	\begin{align*}
		x(s)
			&= \sup_{\substack{t > s\\t \in J}} y(t)
			= \sup_{\substack{t > s\\t \in J}}
				\sup_{\substack{t' > t\\t' \in J}} y(t')
			= \sup_{\substack{t > s\\t \in J}} x(t)
			= \sup_{t > s} x(t).
	\end{align*}
	The final equality holds because $x$ is decreasing.
	Then $\sup_{t > s} x(t) = \lim_{t \to s^+} x(t)$, meaning that
	$x$ is right-continuous.
	
	We now show that $x = y$ a.e..
	
	It follows from the definition of $x$ that
	\begin{align}\label{e:MgeLs}
		x(s) \ge y(t)
			\text{ for all } 0 \le s < t,
	\end{align}
	and from \cref{e:dec,e:Msup} that
	\begin{align}\label{e:LgeMst}
		y(s) \ge x(s) \text{ for all } s \in J
	\end{align}
	and for all $s, s' \in J$ with $s > 0$ and $0 \le s' < s$,
	\begin{align*}
		x(s') \ge y(s) \ge x(s).
	\end{align*}
	We see from this that $x = y$ at every positive point of continuity of $x$.
	But being monotonic, $x$ has only a countable number of discontinuities,
	so $x = y$ a.e.; hence, $x = y$ on $J'$ for some full measure set $J' \subseteq J$.
	It follows that
	\begin{align*}
		x(t) \le x(s) - K \int_s^t x(\tau) \dd \tau
	\end{align*}
	for all $(s, t)$ in
	\begin{align*}
		A := \set{(s, t) \colon s \in J', t \in J', t > s}.
	\end{align*}
	Hence, for all $(s, t) \in A$,
	\begin{align}\label{e:MBasicIneq}
		\frac{x(t) - x(s)}{t - s}
			&\le -K \frac{1}{t - s} \int_s^t x(\tau) \dd \tau.
	\end{align}
	
	For any $s \in J'$, we will take $t \to s^+$, $t \in J'$ for both sides of \cref{e:MBasicIneq}.
	For the left side,
	\begin{align*}
		\lim_{\substack{t \to s^+\\t \in J'}} \frac{x(t) - x(s)}{t - s}
			&= \lim_{t \to s^+} \frac{x(t) - x(s)}{t - s}
			= x'(s) \text{ a.e.}.
	\end{align*}
	The first equality holds whenever the second limit exists, and the second limit,
	which is the right-derivative of $x(s)$, exists and equals $x'(s)$ a.e., since
	$x$ is monotonic.
	
	For the right side of \cref{e:MBasicIneq},
	\begin{align*}
		\lim_{\substack{t \to s^+\\t \in J'}}
				\frac{1}{t - s} \int_s^t x(\tau) \dd \tau
			&= \lim_{t \to s^+}
				\frac{1}{t - s} \int_s^t x(\tau) \dd \tau
			= x(s) \text{ everywhere},
	\end{align*}
	where we used the right-continuity of $x$ to obtain the limit everywhere.

	We conclude that
	\begin{align*}
		x'(s) \le - K x(s) \text{ a.e.}.
	\end{align*}
	
	Let $s_0 = \inf\set{s \ge 0 \colon x(s) = 0}$, setting $s_0 = \iny$ if $x$ never
	vanishes. Because $x$ is decreasing, $x(s) = 0$ for all $s \ge s_0$.
	
	Let $I = [0, s_0 - \eps]$ for arbitrary $\eps \in (0, s_0)$. Then $x$ is bounded away
	from zero on $I$, so for almost all $s \in I$
	\begin{align*}
		(\log x)'(s) \le - K.
	\end{align*}
	
	Now, $\log x$ is
	decreasing, so by \cref{L:RightContDec},
	\begin{align*}
		\log x(t) - \log x(0)
			\le \int_0^t (\log x)'(s) \dd s
			\le - K t,
	\end{align*}
	from which $x(t) \le x(0) e^{-K t}$ follows for all $t \in I$, and hence,
	in fact, for all $t \in [0, s_0)$ and then for all $t \ge 0$.
	
	Because $0 \in J$, using \cref{e:LgeMst}, we have,
	\begin{align*}
		x(t)
			\le x(0) e^{-K t}
			\le y(0) e^{-K t}.
	\end{align*}
	Then by \cref{e:MgeLs}, for $0 \le s < t$,
	\begin{align*}
		y(t)
			\le x(s)
			\le y(0) e^{-K s}.
	\end{align*}
	Since this holds for all $0 \le s < t$ it follows that $y(t) \le y(0) e^{-K t}$ for
	all $t \ge 0$.
\end{proof}

\begin{lemma}\label{L:RightContDec}
	Let $f$ be
	decreasing on $[0, s_0)$. Then
	for all $[a, b] \subseteq [0, s_0)$,
	\begin{align*}
		f(b) - f(a)
			&\le \int_a^b f'(s) \dd s.
	\end{align*}
\end{lemma}
\begin{proof}
	See, for example, Theorem 3 Chapter 5 of \cite{Royden}, adapted to decreasing
	rather than increasing functions.
\end{proof}

\scriptsize{\section*{Acknowledgments} Part of this work was prepared while Kelliher was participating in a program hosted by the Simons Laufer Mathematical Sciences Research Institute in Berkeley, California, in Spring 2021 and again in Summer 2023, supported by the National Science Foundation under Grant No. DMS-1928930. The second author is partially supported by the French National Research Agency in the framework of the project ``BOURGEONS'' (ANR-23-CE40-0014-01) and ``ComplexFlows'' of the PEPR MathsViVEs (ANR-23-EXMA-0004). The third and fourth authors gratefully acknowledge the hospitality of the Department of Mathematics at the University of California, Riverside, and of the Institut Fourier at the Universit\'e de Grenoble, where part of this research was done. MCLF was partially supported by CNPq, through grant \# 304990/2022-1, and FAPERJ, through grant \# E-26/201.209/2021.
HJNL acknowledges the support of CNPq, through grant \# 305309/2022-6, and of FAPERJ, through grant \# E-26/201.027/2022. The work of E.S.T. was partially supported by the DFG Research Unit FOR 5528 on Geophysical Flows. In addition, MCLF and HJNL would like to thank the Isaac Newton Institute for Mathematical Sciences for support and hospitality during the program ``Mathematical aspects of turbulence: where do we stand'', when part of the work on this paper was undertaken. This work was
supported in part by: EPSRC Grant Number EP/R014604/1.}


\providecommand{\MR}[1]{}\renewcommand{\MR}[1]{}\def\cprime{$'$}
  \def\polhk#1{\setbox0=\hbox{#1}{\ooalign{\hidewidth
  \lower1.5ex\hbox{`}\hidewidth\crcr\unhbox0}}}

\end{document}